\documentclass[onefignum,onetabnum]{siamart190516}

\usepackage{lipsum}
\usepackage{amsfonts,amssymb}
\usepackage{graphicx}
\usepackage{epstopdf}
\usepackage{algorithmic}
\usepackage{color}
\usepackage{ulem}
\usepackage{comment}
\usepackage{bm}
\usepackage{tikz}

\usepackage{booktabs} 

\newcommand{\R}{\mathbb{R}}

\usepackage{framed}

\ifpdf
  \DeclareGraphicsExtensions{.eps,.pdf,.png,.jpg}
\else
  \DeclareGraphicsExtensions{.eps}
\fi

\newcommand{\algorithmicbreak}{\textbf{break}}
\newcommand{\BREAK}{\STATE \algorithmicbreak}

\newcounter{algnum}
\makeatletter
	\renewcommand{\thealgorithm}{
	\arabic{algnum}-(\arabic{algorithm})}
	\@addtoreset{algorithm}{algnum}
\makeatother

\newsiamremark{remark}{Remark}
\newsiamremark{hypothesis}{Hypothesis}
\crefname{hypothesis}{Hypothesis}{Hypotheses}
\newsiamthm{claim}{Claim}
\newsiamthm{example}{Example}
\newtheorem{problem}{Problem}
\headers{The Nearest Graph Laplacian in Frobenius Norm}{K. SATO AND M. SUZUKI}

\title{The Nearest Graph Laplacian in Frobenius Norm\thanks{Submitted to the editors DATE.
\funding{This work was supported by Japan Society for the Promotion of Science KAKENHI under 23K03899.}}}

\author{Kazuhiro Sato\thanks{Department of Mathematical Informatics, Graduate School of Information Science and Technology, the University of Tokyo,
  (K. Sato: \email{kazuhiro@mist.i.u-tokyo.ac.jp}, M. Suzuki: \email{suzuki-masato602@g.ecc.u-tokyo.ac.jp})}
\and Masato Suzuki\footnotemark[2]}

\usepackage{amsopn}

\ifpdf
\hypersetup{
  pdftitle={An Example Article},
  pdfauthor={D. Doe, P. T. Frank, and J. E. Smith}
}
\fi

\begin{document}

\maketitle
\begin{abstract}
We address the problem of finding the nearest graph Laplacian to a given matrix, with the distance measured using the Frobenius norm. Specifically, for the directed graph Laplacian, we propose two novel algorithms by reformulating the problem as convex quadratic optimization problems with a special structure: one based on the active set method and the other on direct computation of Karush-Kuhn-Tucker (KKT) points.
The proposed algorithms can be applied to system identification and model reduction problems involving Laplacian dynamics.
We demonstrate that these algorithms possess lower time complexities and the finite termination property, unlike the interior point method and V-FISTA, the latter of which is an accelerated projected gradient method.
Our numerical experiments confirm the effectiveness of the proposed algorithms.
\end{abstract}

\begin{keywords}
  Convex Quadratic Optimization, Directed Graph, Graph Laplacian, Optimization Algorithm
\end{keywords}

\begin{AMS}
  	05C50, 05C20, 90C20
\end{AMS}

\section{Introduction}
\label{sec: Introduction}
\subsection{Background}
\label{subsec: Background}
The graph Laplacian matrix is a fundamental tool in graph theory and has important applications in various fields \cite{bapat2010graphs, sugiyama2023kron, vishnoi2013lx}. 
Networked systems can be found in many places in the real world, and dynamical behaviors on networks are often modeled with the graph Laplacian dynamics \cite{Mirzaev2013-jp}
\begin{equation}
\label{lap_dynamics}
\dot{x}(t)=-Lx(t),
\end{equation}
where $x(t)\in \mathbb{R}^n$ and $L\in \mathbb{R}^{n\times n}$ denote the state vector at time $t$ and the graph Laplacian, respectively. Examples of (\ref{lap_dynamics}) are seen in social networks \cite{Proskurnikov2017-up, Yi2021-to, Yi2019-iu}, multi-agent systems \cite{bullo2019lectures,Mesbahi2010-tu,Olfati-Saber2007-tk,Ren2007-ms}, biochemical reaction systems \cite{Ahsendorf2014-gz, Estrada2016-hv, Gunawardena2012-sh, Karp2012-yz}, synchronization systems \cite{Ashwin2016-nw,Tuna2016-bk}, and brain networks \cite{abdelnour2018functional, tang2020control}.
An equivalent expression of (\ref{lap_dynamics}) is
\begin{equation*}
\dot{x}_i(t) = \sum_{j}w_{ij}(x_j(t)-x_i(t)),
\end{equation*}
where $w_{ij}=-L_{ij}$ is a non-negative value. The $i$-th state $x_i(t)$ increases or decreases to minimize the difference between the $j$-th state $x_j(t)$, with the reference weight $w_{ij}$. The dynamics reach an equilibrium state with the same value for each variable, which is called a consensus \cite{bullo2019lectures}.

\subsubsection{System identification problem} \label{Intro_identification}

The various applications of graph Laplacians give rise to the system identification problem of such dynamics on graphs; in other words, the problem of identifying the graph structure and the edge weights of the network from observed data. In such situations, we need to construct a graph Laplacian from a noisy matrix that may not necessarily be a graph Laplacian.
In more detail, we can consider
a discretized model system \eqref{lap_dynamics} with noise $\varepsilon_k$, using the Euler method, as described by
\begin{align}
    x_{k+1} = (\mathrm{I}_n-hL)x_k + \varepsilon_k, \label{discrete_equation}
\end{align}
where $h>0$ is the sampling interval.
Using \eqref{discrete_equation}, we can formulate the system identification problem as
\begin{equation}
\begin{aligned}
\min_{L\in \mathbb{R}^{n\times n}} \quad & \frac{1}{N} \|X'-(\mathrm{I}_n-hL)X\|_{\rm F}^2\\
\mathrm{subject\ to}\quad& L
\mathrm{\,\, is\,\,a\,\,graph\,\, Laplacian},
\end{aligned} \label{problem_identification}
\end{equation}
where
\begin{align*}
X:=\begin{bmatrix}
    x_0 & x_1 & \cdots & x_{N-1}
\end{bmatrix} \in  \mathbb{R}^{n\times N},\quad X':=\begin{bmatrix}
    x_1 & x_2 & \cdots & x_{N}
\end{bmatrix} \in \mathbb{R}^{n\times N}.
\end{align*}
A similar problem formulation can be found in \cite{obara2024stable}.

Because the constraint in \eqref{problem_identification} can be expressed as a closed convex set, as explained in Section \ref{Sec_Graphs}, the problem \eqref{problem_identification} can be solved using a projected gradient method onto the constraint.
In this method, we need to iteratively solve
\begin{align}
\begin{aligned}
\min_{L\in \mathbb{R}^{n\times n}} \quad & \|A-L\|_{\rm F}^2\\
\mathrm{subject\ to}\quad& L
\mathrm{\,\, is\,\,a\,\,graph\,\, Laplacian},
\end{aligned} \label{problem_projection}
\end{align}
where $A\in \mathbb{R}^{n\times n}$ is a given matrix.
 Detailed formulations of \eqref{problem_projection} are presented in Sections \ref{Sec_pro_formulation} and \ref{Sec_pro_formulation2} of this paper.

\subsubsection{Model reduction problem} \label{Intro_reduction}

Optimization problem \eqref{problem_projection} arises when we consider an $H^2$ optimal model reduction problem.
In fact, suppose that
\begin{align}
    \begin{cases}
        \dot{x}(t) = -Lx(t) + Bu(t)\\
        y(t) =Cx(t)
    \end{cases} \label{system_original}
\end{align}
is a large-scale graph Laplacian dynamics with input $u(t)\in \mathbb {R}^m$ and output $y(t)\in \mathbb{R}^p$.
To facilitate the analysis and control of system \eqref{system_original}, we aim to reduce its size by approximating it with a small-scale graph Laplacian dynamics
\begin{align}
    \begin{cases}
        \dot{x}_r(t) = -L_rx_r(t) + B_ru(t)\\
        y_r(t) =C_rx_r(t),
    \end{cases} \label{system_reduced}
\end{align}
where $L_r\in \mathbb{R}^{r\times r}$ is a reduced graph Laplacian matrix.
That is,
we aim to design $(L_r, B_r, C_r)$ such that the systems \eqref{system_original} and \eqref{system_reduced} are as close as possible in the sense of the $H^2$ norm.
To this end, we can consider iteratively updating $(L_r, B_r, C_r)$ using the cyclic block projected gradient method proposed in \cite{sato2023reduced}.
In this approach, we iteratively solve the optimization problem \eqref{problem_projection}, with $n$ and $L$ replaced by $r$ and $L_r$, respectively.
By adopting this method, we can address the model reduction problem for Laplacian dynamics, as outlined in Section VII of \cite{sato2023reduced} as future work.

\subsection{Objective and Related works}

Therefore, under the assumption that the network structure is known, we develop algorithms to solve the nearest graph Laplacian problem \eqref{problem_projection}.
This assumption is based on the fact that, unlike the identification of nodes and edges, determining the edge weights is challenging in practice due to sensor noise and the lack of quantification methods \cite{Liu2011-rc, terasaki2022minimal, terasaki2024minimal}. A typical example of this setting is in social relation networks, where it is relatively easy to identify the existence of a relationship but difficult to quantify its strength.


It is worth noting that the problem of finding the nearest matrix within certain matrix classes has been well-studied.
In \cite{sato2019optimal}, an efficient and simple algorithm is proposed for solving the nearest graph Laplacian problem in the entry-wise 1-norm.
In \cite{Anderson2017-zy}, the author studied the problem of the nearest stable Metzler matrix. 
In \cite{Bai2007-wb}, the authors studied the nearest doubly stochastic matrix problem. This problem bears some resemblance to ours, as the feasible set is defined by constraints on row sums, column sums, and signatures. The problem of finding the nearest correlation matrix was considered by the authors in \cite{Higham2002-mq, Qi2006-mk}. In \cite{Gillis2017-ze, Orbandexivry2013-dc}, the authors studied the nearest system subject to constraints on its stability properties. However, these methods are not applicable to problem \eqref{problem_projection}, as they may not yield a matrix that is a graph Laplacian.

\subsection{Contribution}
\label{subsec: Contribution}
\begin{itemize}
    \item We reformulate the problem of constructing the nearest loop-less and loopy graph Laplacians, as described in \eqref{problem_projection}, into convex quadratic optimization problems. Theoretical properties arising from the special structures of these problems are examined in detail.
    
   
    \item Using the theoretical properties, we develop an active set algorithm and a direct computation algorithm based on Karush-Kuhn-Tucker (KKT) points for solving our optimization problems.
     The algorithms compute the optimal solutions directly.
    Notably, the active set algorithm can be seen as an application of the proposed algorithm in \cite{Kunisch2003-jz}.
    However, we prove some stronger results specialized for our specific problem. In fact, the authors of \cite{Kunisch2003-jz} showed that the algorithm stops within $2^d$ times of updating the active set, while our case stops within $d$ times of updates.

    

    \item 
  We derive the computational complexities of the proposed algorithms and show that they are more favorable than those of the interior point method and V-FISTA, with the latter being an accelerated projected gradient method. Moreover, we emphasize that, unlike the interior point method and V-FISTA, the proposed methods possess the finite termination property. Through numerical experiments, we demonstrate the effectiveness of the algorithms based on the active set method and direct computation of KKT points. Furthermore, we illustrate that the algorithm based on the active set method encounters some worst-case scenarios where its performance may degrade.

\end{itemize}

\subsection{Outline}
The rest of the paper is organized as follows. In Section \ref{notations}, we introduce notations and define some basic concepts in graph theory. In Section \ref{DSG}, we consider the nearest graph Laplacian problem in the case of directed simple graphs and propose two novel methods. In Section \ref{DG}, we generalize the algorithms to graphs with self-loops. Experimental results are presented in Section \ref{Exp}, and the conclusion is derived in Section \ref{Conclusion}.

\section{Preliminaries}\label{notations}
\subsection{Matrices and Vectors}
Let $\mathbb{R},\ \mathbb{R}_{\geq 0}$, and $\mathbb{N}$ be the set of all real numbers, non-negative real numbers, and positive integers, respectively. For a finite set $X$, let $|X|$ be its cardinality. Let $[n]:=\{1,\dots,n\}$ for $n\in \mathbb{N}$. 
Let $\bm 1_{p\times q}, \bm 0_{p\times q}\in \mathbb{R}^{p\times q}$ be the $p \times q$ dimensional matrices of all ones and zeros, respectively. We use the shorthands $\bm1_d =\bm 1_{d\times 1},\bm 0_d =\bm 0_{d\times 1}$. 
$\mathrm{I}_d$ denotes the $d$-dimensional identity matrix and $\mathrm{J}_d:=\bm 1_{d\times d}$ denotes the $d\times d$ matrix of all ones. We frequently use the matrix $Q_d:=2\mathrm{I}_d+2\mathrm{J}_d$, which is the matrix with $4$ for the diagonal components and $2$ for the non-diagonal components. For any $i,j\in [n]$, let $A_{ij}$ denote the $(i,j)$-element of a matrix $A\in \mathbb{R}^{n\times n}$ and $b_F\in\mathbb{R}^{|F|}$ denote the subvector of $b\in \mathbb{R}^n$ obtained by indices $F\subset [n]$. Given a matrix $A\in \mathbb{R}^{n\times n}$, $A^\top$ and $\mathrm{tr}(A)$ represent the transpose and the trace of $A$, respectively. Let $\|\cdot\|_\mathrm{F}$ denote the Frobenius norm of a matrix, that is, $\|A\|_\mathrm{F}=\sqrt{\mathrm{tr}\left(A^\top A\right)}$. The Frobenius norm can also be seen as the entrywise 2-norm, that is, $\|A\|_{\mathrm{F}}=\sqrt{\sum_{i=1}^n\sum_{j=1}^nA_{ij}^2}$. For same size vectors $b,c\in \mathbb{R}^d$, $b\leq c$ denotes the entry-wise inequality $b_i \leq c_i\ (i=1,\dots,d)$. Let $\mathrm{diag}\{a_1,\dots, a_n\}$ be the diagonal matrix with diagonal components $a_1,\dots,a_n$.

The following proposition is utilized multiple times throughout this paper.

\begin{proposition}[Sherman-Morrison Formula \cite{meyer2000matrix}]\label{Sherman-Morrison}
Suppose $A\in \mathbb{R}^{n\times n}$ is a nonsingular matrix and $u,v\in \mathbb{R}^n$ are vectors. If $A+uv^\top$ is nonsingular, the inverse matrix is,
\begin{equation}
\left(A+uv^\top\right)^{-1}=A^{-1}-\frac{1}{1+v^\top A^{-1}u}A^{-1}uv^\top A^{-1}.
\end{equation}
\end{proposition}

\subsection{Graphs} \label{Sec_Graphs}
Let $G=(V,E,w)$ be a weighted directed graph, with the node set $V = \{1,2,\dots,n\}$, the edge set $E\subseteq V\times V$, and the edge weight function $w: E\to \mathbb{R}_{\geq 0}$. We do not consider any multiedges. The pair $(V,E)$ is called a graph structure. The number of edges will be denoted by $m$.
The directed edge from node $i$ to node $j$ will be denoted by $\{i,j\}$. The edge $\{i,i\}$ denotes the self-loop on node $i$. 
The neighbor set of node $i$ is defined as $\mathcal{N}(i):=\left\{j\in V\mid \{i,j\}\in E,\ i\neq j\right\}$. For any edge $\{i,j\}\in E$, $w_{ij}:=w(\{i,j\})$ shows the edge weight of $\{i,j\}$. We assume that edge weights are nonnegative. The weighted adjacency matrix $W$ is an $n\times n$ matrix with $W_{ij}=w_{ij}$ for any $\{i,j\}\in E$ and $W_{ij}=0$ for $\{i,j\}\notin E$. 
\begin{definition}[simple graph]
A simple graph is a graph with no self-loops (i.e., $\{i,i\}\notin E\ (i=1,\dots,n)$) and no multiedges.
\end{definition}
Using the weighted degree matrix $D := \mathrm{diag}\{D_1,\dots,D_n\}$, where $D_i := \sum_{j=1}^n W_{ij}$ is the weighted degree, the graph Laplacian of $G$ is defined as follows.
\begin{definition}[(loop-less) graph Laplacian]\label{def_loopless}
The loop-less graph Laplacian of a weighted graph $G=(V,E,w)$ is defined as
\begin{equation*}
L:=D-W.
\end{equation*}
\end{definition}
By Definition \ref{def_loopless}, loop-less graph Laplacians of directed graphs satisfy the following (\ref{label_diag_nondiag}) and (\ref{row_sum_zero}).
\begin{itemize}
\item The diagonal elements $L_{ii}$ are non-negative, the non-diagonal elements $L_{ij}$ are non-positive for $\{i,j\}\in E$ and zero for $\{i,j\}\notin E$, i.e.,
\begin{equation}\label{label_diag_nondiag}
L_{ii}\geq 0\ (i=1,\dots,n),\ L_{ij}\leq 0\ (i\neq j,\{i,j\}\in E),\ L_{ij}=0\ (i\neq j,\{i,j\}\notin E).\end{equation}
\item The row-sum is zero, i.e.,
\begin{equation}\label{row_sum_zero}
L\bm 1_n = \bm 0_n.
\end{equation}
\end{itemize}
Conversely, when the graph structure $(V, E)$ is simple, any matrix $L\in \mathbb{R}^{n\times n}$ satisfying properties (\ref{label_diag_nondiag}) and (\ref{row_sum_zero}), uniquely determines the edge weights by defining $w_{ij}=-L_{ij}\ (\{i,j\}\in E)$.
Thus, the set of all directed loop-less graph Laplacians of a simple graph structure $(V, E)$ is defined as
\begin{equation}
\label{def_Lsd}
\mathcal{L}_{sd}(V,E):=\{L\in \mathbb{R}^{n\times n}\mid L\ \mathrm{satisfies}\ (\ref{label_diag_nondiag}),(\ref{row_sum_zero})\}.
\end{equation}

Although the information about self-loops is lost in Definition \ref{def_loopless}, the loopy graph Laplacian for general graphs with self-loops is defined as follows.
\begin{definition}[(loopy) graph Laplacian]
Let $S$ be the self-loop matrix defined as $S:=\mathrm{diag}\{W_{ii},\dots,W_{nn}\}$. The loopy graph Laplacian of a weighted graph $G=(V,E,w)$ is defined as
\begin{equation*}
L:=D-W+S.
\end{equation*}
\end{definition}
Here, loopy graph Laplacians satisfy (\ref{label_diag_ns}) and (\ref{label_diag}) instead of (\ref{row_sum_zero}).
\begin{itemize}
\item If node $i$ does not have a self-loop, the row sum of row $i$ equals zero, i.e.,
\begin{equation}\label{label_diag_ns}
\{i,i\}\notin E \quad\Rightarrow \quad\sum_{j=1}^n L_{ij}= 0.
\end{equation}
\item If node $i$ has a self-loop, the row sum of row $i$ is nonnegative, i.e.,
\begin{equation}\label{label_diag}
\{i,i\}\in E \quad\Rightarrow \quad \sum_{j=1}^n L_{ij}\geq 0.
\end{equation}
\end{itemize}

Similar to the case of simple graphs, the set of all loopy graph Laplacians of a directed graph structure $(V, E)$ is defined as
\begin{equation}
\label{def_Ld}
\mathcal{L}_d(V,E):=\{L\in \mathbb{R}^{n\times n}\mid L\ \mathrm{satisfies}\ (\ref{label_diag_nondiag}),(\ref{label_diag_ns}),(\ref{label_diag})\}.
\end{equation}


\section{Loop-Less Graph Laplacians}\label{DSG}
In this section, we consider the problem of finding the nearest graph Laplacian to a given arbitrary matrix, in the case of loop-less graph Laplacians of simple directed graphs. The problem is reformulated as a convex quadratic optimization problem with non-positivity constraints. 
To solve the problem, we propose two efficient algorithms, and we prove some properties of the algorithms.

\subsection{Problem Formulation} \label{Sec_pro_formulation}
Let $G=(V, E, w)$ be a weighted simple directed graph, i.e., $\{i,i\}\notin E$ for any $i=1,2,\dots,n$. 


Our purpose is to reconstruct the graph Laplacian from a ``noisy'' Laplacian matrix by finding the nearest matrix that satisfies the conditions of a graph Laplacian of $(V, E)$. 
Because
the set of matrices that satisfy the conditions of directed loop-less graph Laplacians is defined in (\ref{def_Lsd}),
our problem can be formulated as follows.
\begin{framed}
\begin{problem}
\label{prob_1}
Given a graph structure $(V,E)$ and a matrix $A\in \mathbb{R}^{n\times n},$
\begin{equation*}
\begin{aligned}
\min_{L\in \mathbb{R}^{n\times n}} \quad & \|A-L\|_\mathrm{F}^2\\
\mathrm{subject\ to}\quad& L\in \mathcal{L}_{sd}(V,E).
\end{aligned}
\end{equation*}
\end{problem}
\end{framed}

Note that
we assume that the graph structure $(V, E)$ is known, and
minimizing $\|A-L\|_{\mathrm{F}}^2$ is equivalent to minimizing $\|A-L\|_{\mathrm{F}}$.


Problem \ref{prob_1} is a convex quadratic problem with linear equality and inequality constraints and could be solved by quadratic solvers or convex solvers. However, the equality constraint is relatively difficult to tackle and a better formulation could be derived by exploiting the structure of our specific problem.

First, we derive an equivalent convex quadratic optimization problem only with non-positivity constraints.
For any $A\in \mathbb{R}^{n\times n}$ and $L\in \mathcal{L}_{sd}(V,E)$,
\begin{align}
\|A-L\|_{\mathrm{F}}^2&=\sum_{i=1}^n(A_{ii}-L_{ii})^2+\sum_{\{i,j\}\in E}(A_{ij}-L_{ij})^2+\sum_{\{i,j\}\notin E,\,i\neq j}(A_{ij}-0)^2\nonumber\\
&=\sum_{i=1}^n\sum_{j=1}^nA_{ij}^2+\sum_{i=1}^n\left\{L_{ii}^2-2A_{ii}L_{ii}+\sum_{j\in \mathcal{N}(i)}(L_{ij}^2-2A_{ij}L_{ij}) \right\}.\nonumber
\end{align}
Since \eqref{row_sum_zero} holds, the diagonal element of $L$ can be written as
$L_{ii} = -\sum_{j\in \mathcal{N}(i)} L_{ij}\quad(i=1,\dots,n)$.
Thus, the objective function of Problem \ref{prob_1} can be rewritten as
\begin{align*}
&\|A-L\|_{\mathrm{F}}^2\nonumber\\
&=\|A\|_\mathrm{F}^2+\sum_{i=1}^n\left\{\left(-\sum_{j\in \mathcal{N}(i)}L_{ij} \right)^2-2A_{ii}\left(-\sum_{j\in \mathcal{N}(i)}L_{ij} \right)+\sum_{j\in \mathcal{N}(i)}(L_{ij}^2-2A_{ij}L_{ij}) \right\}\nonumber\\
&=\|A\|_\mathrm{F}^2+\sum_{i=1}^n\left\{\sum_{j\in \mathcal{N}(i)}2L_{ij}^2+\sum_{j,k\in \mathcal{N}(i), j\neq k}2L_{ij}L_{ik}+\sum_{j\in \mathcal{N}(i)}(2A_{ii}-2A_{ij})L_{ij}\right\}\nonumber\\
&=\frac{1}{2}x^\top Q x + b^\top x+\|A\|_{\mathrm{F}}^2, 
\end{align*}
where
\begin{align}
\label{def_Q}
Q &:= \begin{bmatrix}Q_{d_1}&&\\&\ddots & \\ &&Q_{d_n}\end{bmatrix},\ 
Q_d := \begin{bmatrix}4&2&\dots& 2\\2&4 &2&\vdots \\ \vdots&&\ddots&2\\2&\dots&2&4\end{bmatrix}=2\mathrm{I}_d+2\mathrm{J}_d,\\
d_i&:=|\mathcal{N}(i)|,\quad
x:=\begin{bmatrix}L_{i_1j_1},L_{i_2j_2},\ldots, L_{i_mj_m}\end{bmatrix}^\top,\nonumber \\
\label{def_b}b&:=\begin{bmatrix}2A_{i_1i_1}-2A_{i_1j_1},\dots,2A_{i_mi_m}-2A_{i_mj_m}\end{bmatrix}^\top.\nonumber
\end{align}
Here, $\{i_k,j_k\}$ is the $k$-th edge in $E$, and $d_1+\cdots +d_n =m$.
Since the elements in $x$ are the non-diagonal elements in $L$, every element in $x$ must be nonpositive. Thus, Problem \ref{prob_1} can be reformulated as follows.
\begin{framed}
\begin{problem}
\label{prob_2}
\begin{equation*}
\begin{aligned}
\min_{x \in \mathbb{R}^{m}} \quad & \frac{1}{2}x^\top Q x + b^\top x\\
\mathrm{subject\ to} \quad& x\leq \bm 0_m.
\end{aligned}
\end{equation*}
\end{problem}
\end{framed}
Since $Q$ defined as \eqref{def_Q} is a block diagonal matrix with $n$ blocks, we can divide Problem \ref{prob_2} into $n$ smaller problems, where each block corresponds to each row of $L$. The one-block problem of the $i$-th row $(i=1,\dots,n)$ is formulated as follows.
\begin{framed}
\begin{problem}[one-block problem of row $i$]
\label{prob_3}
\begin{equation*}
\begin{aligned}
\min_{x \in \mathbb{R}^{d_i}} \quad & \frac{1}{2}x^\top Q_{d_i} x + b_{(i)}^\top x\\
\mathrm{subject\ to} \quad& x\leq \bm 0_{d_i}.
\end{aligned}
\end{equation*}
\[Q_{d_i} := 2\mathrm{I}_{d_i}+2\mathrm{J}_{d_i},\ b_{(i)} := \begin{bmatrix}b_{d_1+\dots,+d_{i-1}+1},\dots,b_{d_1+\dots+d_i}\end{bmatrix}^\top.\]
\end{problem}
\end{framed}

We focus on Problem \ref{prob_3} in the following sections. When there is no need to specify the row, we might omit the notation $i$, such as $Q_d$ instead of $Q_{d_i}$, and $b$ instead of $b_{(i)}$.


An optimal solution to Problem \ref{prob_1} uniquely exists,
because $Q_d$ is a symmetric positive definite matrix as shown in the following.

\begin{lemma}\label{lem_Qd_spd}
For any $d\in \mathbb{N}$, $Q_d$ is a symmetric positive definite matrix with eigenvalues $2+2d$ (with multiplicity $1$) and $2$ (with multiplicity $d-1$). 
The inverse matrix of $Q_d$ is,
\begin{equation}
\label{Qd_inv}
Q_d^{-1} = \frac{1}{2}\left(\mathrm{I}_d-\frac{1}{1+d}\mathrm{J}_d\right).
\end{equation}
\end{lemma}
\begin{proof}
$\mathrm{J}_d=\bm1_d\bm 1_d^\top$ is a symmetric matrix with $\mathrm{rank}(\mathrm{J}_d)=1$. The number of nonzero eigenvalues of a symmetric matrix is equal to its rank and hence $\mathrm{J}_d$ has only one nonzero eigenvalue, that is, eigenvalue $d$ with eigenvector $\bm 1_d$. Thus, the eigenvalues of $Q_d=2\mathrm{I}_d+2\mathrm{J}_d$ are $2+2d$ with multiplicity $1$ and $2$ with multiplicity $d-1$.
Thus, $Q_d$ is nonsingular, and by the Sherman-Morrison formula (Proposition \ref{Sherman-Morrison}), we have \eqref{Qd_inv}.
\end{proof}

Problem \ref{prob_3} can be solved using
the primal-dual active set algorithm proposed by Kunisch and Rendl \cite{Kunisch2003-jz}.
In Theorem 3.4 of \cite{Kunisch2003-jz}, Kunisch and Rendl showed that the update of the active set occurs at most $2^{d_i}$ times.
In the next section, we prove that our active set algorithm stops after at most $d_i$ updates thanks to the special structure of $Q_{d_i}$.
It should be remarkable that our analysis in the next section is considerably different from that in \cite{Kunisch2003-jz}.

\subsection{Proposed Algorithm 1} \label{Sec_Alg1}
 In this section, we construct an iterative algorithm for solving Problem \ref{prob_3} that can be seen as an active set method \cite{More1989-fm, Kunisch2003-jz}. The main idea of active set methods is dividing the inequality constraints into two sets, the active set $B \subseteq [d]$ and the free set $F=[d]\backslash B$. We fix the variables in the active set constraints onto the constraint bound (i.e., $x_i=0$ for any $i\in B$), and then solve the unconstrained optimization problem by ignoring the constraints in $F$. An optimal active set leads to the optimal solution of the original optimal solution, and therefore we need to construct a good approximation and updating method of the active set.

Let $x^{\ast}\in \mathbb{R}^d$ be the optimal solution to Problem \ref{prob_3}. 
Using $b':=-Q_d^{-1}b$, $x^\ast$ can be seen as the minimizer of
\[g(x):=\frac{1}{2}(x-b')^\top Q_d(x-b')=\frac{1}{2}x^\top Q_d x + b^\top x+ \mathrm{constant.}\]
over $C_d$,
where
\begin{align}
C_d := \{x\in \mathbb{R}^d\mid x_i \leq 0\ (i=1,\dots,d)\}. \label{C_d}
\end{align}
Thus, $b'$ is the optimal solution of the unconstrained version of Problem \ref{prob_3}.

\begin{lemma}\label{thm_grad_newton}
$b':=-Q_d^{-1}b$ can be calculated in $O(d)$ time.
\end{lemma}
\begin{proof}
From \eqref{Qd_inv},
\begin{align*}
    b'=-\frac{1}{2}b+\frac{1}{2}\frac{1}{1+d}\left(\sum_{i=1}^d b_i\right)\bm 1_d.
\end{align*}
Thus, $b'$ can be calculated in $O(d)$ time.
\end{proof}

By definition of $Q_d$, 
\begin{align}
g(x) = \|x-b'\|_2^2+\left(\bm 1^\top (x-b')\right)^2,    \label{eq_g}
\end{align}
which is used in the following.

\begin{lemma}
\label{lem_ineq}
Inequality $\bm{1}^\top x^\ast \leq \bm 1^\top b'$ holds.
\end{lemma}
\begin{proof}
If $\bm 1^\top b'\geq 0$, $x^\ast\in C_d$ yields 
$\bm 1^\top x^\ast \leq 0\leq \bm 1^\top b'$. 
Thus, we assume $\bm 1^\top b'< 0$. Suppose that
\begin{equation}\label{assumption_lem}\bm{1}^\top x^\ast > \bm 1^\top b'.\end{equation}
Let $x'$ be the orthogonal projection of $x^\ast$ onto the hyperplane $\{x\in \mathbb{R}^d \mid \bm 1 ^\top x = \bm 1^\top b'\}$. Because the explicit form of the orthogonal projection $x'$ is
$x'=x^\ast - \frac{1}{d}\left(\bm 1^\top x^\ast-\bm 1^\top b'\right)\bm 1$,
assumption (\ref{assumption_lem}) implies that each element of $x'$ is smaller than $x^\ast$, and thus $x'$ is a feasible point (i.e. $x'\in C_d$).
From \eqref{eq_g},
\begin{align}
g(x^\ast)-g(x')
\label{eq_IJ}
=\|x^\ast-b'\|_2^2+\left(\bm 1^\top (x^\ast-b')\right)^2-\|x'-b'\|_2^2-\left(\bm 1^\top (x'-b')\right)^2.
\end{align}
By the $d$-dimensional Pythagorean theorem, we have that
\begin{equation}\label{pytha}\|x^\ast-b'\|_2^2=\|x^\ast -x'\|_2^2+\|x'-b'\|_2^2.\end{equation}
Combining (\ref{eq_IJ}), (\ref{pytha}), and $\bm 1^\top x'=\bm 1^\top b'$, 
$g(x^\ast)-g(x')
>0$.
 This contradicts the fact that $x^\ast$ is the minimizer of $g$ over $C_d$.
\end{proof}
\begin{theorem}
\label{main_thm1}
For any $i = 1,\dots, d$, if $b'_i$ is positive, then $x_i^\ast=0$.
\end{theorem}
\begin{proof}
Suppose that there exists an index $k\in\{1,\dots, d\}$ that satisfies $b'_k > 0$ and $x^\ast_k<0$. From Lemma \ref{lem_ineq}, it is sufficient to consider the following two cases. In the cases, we derive a contradiction by constructing a feasible point $x'$ that satisfies $g(x^\ast)-g(x')>0$.\\
(i) If $\bm 1^\top x^\ast = \bm 1^\top b'$, let $x'\in C_d$ be 
$x'_i = \begin{cases}0&(i=k)\\ x_i^\ast & (i\neq k)\end{cases}$.
From \eqref{eq_g},
\begin{align}
\label{eq_thm1_1}
g(x^\ast)-g(x')&=\|x^\ast-b'\|_2^2+\left(\bm 1^\top (x^\ast-b')\right)^2-\|x'-b'\|_2^2-\left(\bm 1^\top (x'-b')\right)^2.
\end{align}
From the definition of $x'$, we have
\begin{align}
\|x^\ast-b'\|_2^2-\|x'-b'\|_2^2&=\sum_{i=1}^d \left(x^\ast_i - b'_i\right)^2-\sum_{i=1}^d \left(x'_i-b'_i\right)^2\nonumber\\
&= \left(x^\ast_k-b'_k\right)^2-\left(0-b_k'\right)^2\nonumber\\
\label{eq_thm1_2}
&= {x_k^\ast}^2-2b'_kx_k^\ast.\\
\left(\bm 1^\top (x^\ast-b')\right)^2-\left(\bm 1^\top (x'-b')\right)^2&=\left(\bm 1^\top (x^\ast-b')\right)^2-\left(\bm 1^\top (x^\ast-b')-x_k^\ast\right)^2\nonumber\\
&= 0-(0-x_k^\ast)^2\nonumber\\
&=-{x_k^\ast}^2\label{eq_thm1_3}
\end{align}
Combining (\ref{eq_thm1_1}), (\ref{eq_thm1_2}), and (\ref{eq_thm1_3}),
$g(x^\ast)-g(x')=-2b'_kx_k^\ast>0$.

\noindent
(ii) If $\bm 1^\top x^\ast < \bm 1^\top b'$,  let $x'\in C_d$ be 
$x'_i = \begin{cases}\min\{0, x_k^\ast+\bm 1^\top b'-\bm 1^\top x^\ast\}&(i=k)\\ x_i^\ast & (i\neq k)\end{cases}$.
Similar to case (i), we have
\begin{align}
g(x^\ast)-g(x')
\label{eq_thm1_4}
= \left(x_k^\ast-b'_k\right)^2-\left(x'_k-b'_k\right)^2+\left(\bm 1^\top (x^\ast-b')\right)^2-\left(\bm 1^\top (x^\ast-b')-x_k^\ast+x_k'\right)^2.
\end{align}
From $x_k^\ast < x_k'\leq 0 <b_k'$, we have
\begin{equation}
\label{eq_thm1_5}
\left(x_k^\ast-b'_k\right)^2-\left(x'_k-b'_k\right)^2>0.
\end{equation}
Since $x_k^\ast <x'_k \leq x_k^\ast +\bm 1^\top b' -\bm 1^\top x^\ast$, we have
$\bm 1^\top x^\ast-\bm 1^\top b'< \bm 1^\top x^\ast-\bm 1 ^\top b'-x_k^\ast+x_k'\leq 0$,
and therefore, 
\begin{align}
\label{eq_thm1_6}
\left(\bm 1^\top (x^\ast-b')\right)^2-\left(\bm 1^\top (x^\ast-b')-x_k^\ast+x_k'\right)^2>0.
\end{align}
Combining (\ref{eq_thm1_4}), (\ref{eq_thm1_5}), and (\ref{eq_thm1_6}), we derive
$g(x^\ast)-g(x')>0$.
\end{proof}

Theorem \ref{main_thm1} provides a useful characterization of Problem \ref{prob_3}. 
In fact, this theorem asserts that, if the $i$-th element of the unconstrained minimizer $b'$ violates the constraint, we should fix the $i$-th element onto the boundary $0$. \par
Now, let us assume that $b'_{B}>\bm 0$ holds for indices $B\subseteq [d]$. From Theorem \ref{main_thm1}, we can solve Problem \ref{prob_3} with $x_i=0$ for every $i \in B$.
Denoting the left free variables by $F:=[d]\backslash B$, Problem \ref{prob_3} is,
\begin{align*}
\min_{x \in \mathbb{R}^{|F|}} \quad & \frac{1}{2}\begin{bmatrix}x^\top& \bm 0_{|B|}^\top\end{bmatrix} \begin{bmatrix}Q_{|F|}&2\cdot\bm 1_{|F|\times|B|}\\2\cdot\bm 1_{|B|\times|F|}&Q_{|B|}\end{bmatrix} \begin{bmatrix}x\\ \bm 0_{|B|}\end{bmatrix} + \begin{bmatrix}b_F^\top & b_B^\top \end{bmatrix} \begin{bmatrix}x\\ \bm 0_{|B|}\end{bmatrix}\\
\mathrm{subject\ to} \quad& x\leq \bm 0_{|F|}.
\end{align*}
This problem is equivalent to the following problem with the same form of Problem \ref{prob_3}, but smaller in size.
\begin{framed}
\begin{problem}
\label{prob_4}
\begin{align*}
\min_{x \in \mathbb{R}^{|F|}} \quad & \frac{1}{2}x^\top Q_{|F|}x +b_F^\top x\\
\mathrm{subject\ to} \quad& x\leq \bm 0_{|F|}.
\end{align*}
\end{problem}
\end{framed}

Note that the unconstrained optimizer $-Q_{|F|}^{-1}b_{F}$ of Problem 4 does not always satisfy the bound constraint. Now, let $x^\ast\in \R^{|F|}$ be an unconstrained optimizer of Problem 4.  If $x^\ast\leq \bm0_{|F|}$ holds, $x^\ast$ is the optimizer of Problem 4 and $x\in \R^d$ with $x_F=x^\ast$ and $x_B=\bm0_{|B|}$ is the optimizer of Problem 3. Otherwise, (if $x^\ast\leq \bm0_{|F|}$ does not hold,) from Theorem \ref{main_thm1}, we can again add the violating indices to the active set $B$, fix the active set variables to $0$ and reformulate a further reduced version of the problem. 

By repeating the above argument until the unconstrained optimizer satisfies the bound constraint, we can derive Algorithm \ref{alg_1}, which is an active-set type algorithm. The main procedure of Algorithm \ref{alg_1} is:
\begin{description}
\setlength{\leftskip}{1.0cm}
\item[Step 1.] Solve the unconstrained problem with the free variables.
\item[Step 2.] If there are no violations of the constraints, terminate.
\item[Step 3.] Add the indices of the variables that violate the constraints to the active set.
\item[Step 4.] Return to Step 1.
\end{description}

\stepcounter{algnum}
\begin{algorithm}[t]
\begin{algorithmic}[1]
\REQUIRE $d\in \mathbb{N}$, $b\in \mathbb{R}^d$ 
\STATE $x^\ast\leftarrow \bm 0_{d}$
\STATE $B_1 \leftarrow \emptyset,\ F_0\leftarrow \emptyset,\ F_1\leftarrow \{1,\dots, d\}$, $k\leftarrow1$
\WHILE{$|F_{k-1}|\neq |F_{k}|$}
\STATE $y \leftarrow -Q_{|F_k|}^{-1}b_{F_k}$
\STATE $B_{k+1}\leftarrow B_k,\ F_{k+1}\leftarrow F_k$
\FOR{$i\in F_k$}
\IF{$y_i>0$}
\STATE $B_{k+1}\leftarrow B_{k+1}\cup \{i\},\quad F_{k+1}\leftarrow F_{k+1}\setminus \{i\}$
\ENDIF
\ENDFOR
\STATE $k\leftarrow k+1$
\ENDWHILE
\STATE $x^\ast_{F_k}\leftarrow y$
\ENSURE $x^\ast$
\end{algorithmic}
\caption{Active Set Algorithm (for Problem \ref{prob_3})}
\label{alg_1}
\end{algorithm}

\begin{theorem}
\label{alg1_terminate}
Algorithm \ref{alg_1} terminates in finite steps and provides the global optimal solution to Problem \ref{prob_3}.
\end{theorem}
\begin{proof}
If there are no free variables left, the algorithm terminates returning $x=\bm 0_d$. The condition of the while loop remains true if and only if the number of free variables has changed. The number of free variables is non-increasing, and therefore, the while loop will terminate in finite, namely $d$, steps.

From Theorem \ref{main_thm1} and the following discussion, adding the violating indices to the active set and reducing the problem is justified, which indicates that the temporary free set $F_k$ is always a superset of the optimal free set. Since the unconstrained optimizer is always smaller or equal to the constrained optimizer, if the unconstrained optimizer satisfies the non-positivity constraints, the unconstrained optimizer is also the optimizer of the original constrained problem. Because the while loop terminates only if the unconstrained minimizer of the reduced problem does not violate the non-positivity constraint, Algorithm \ref{alg_1} returns the optimal solution to Problem \ref{prob_3}.
\end{proof}

From Theorems \ref{thm_grad_newton} and \ref{alg1_terminate}, the computational complexity of Algorithm \ref{alg_1} is given as follows.
\begin{corollary}\label{com_alg_1}
Algorithm \ref{alg_1} solves Problem \ref{prob_3} with $O(d_i^2)$ time. 
\end{corollary}
\begin{proof}
From Lemma \ref{thm_grad_newton}, each iteration of the while loop costs $O(d_i)$ time. From the proof of Theorem \ref{alg1_terminate}, the while loop repeats $d_i$ times at most. Hence, the computational cost of Problem \ref{prob_3} is $O(d_i^2)$. 
\end{proof}

Although the output $x^\ast$ in Algorithm \ref{alg_1} corresponds to the essential parts of the optimal graph Laplacian, it is not a solution for Problem \ref{prob_1}.
Therefore,
we present Algorithm \ref{alg_1_full} as the complete version for solving the nearest graph Laplacian problem (Problem \ref{prob_1}), taking input $A$ and producing output $L^\ast$.
The following corollary is a consequence of Corollary \ref{com_alg_1}.

\begin{corollary}\label{com_alg_1_full}
Algorithm \ref{alg_1_full} solves Problem \ref{prob_1} with $O(\sum_{i=1}^n d_i^2)$ time.
\end{corollary}

\begin{algorithm}[t]
\begin{algorithmic}[1]
\REQUIRE $A\in \mathbb{R}^{n\times n}$, neighbor set of each node: $\mathcal{N}(i)~(i=1\dots,n)$ 
\STATE $L^\ast\leftarrow \bm 0_{d\times d}$
\FOR{$i=1,\dots,n$}
\STATE $d\leftarrow |\mathcal{N}(i)|$, $B_1 \leftarrow \emptyset,\ F_0\leftarrow \emptyset,\ F_1\leftarrow \{1,\dots, d\}$, $k\leftarrow1$
\STATE $x\leftarrow \bm{0}_d$, $b\leftarrow \begin{bmatrix}2A_{ii}-2A_{iN(i)_1},\dots, 2A_{ii}-2A_{iN(i)_{d}}\end{bmatrix}^\top$
\WHILE{$|F_{k-1}|\neq |F_{k}|$}
\STATE $y \leftarrow -Q_{|F_k|}^{-1}b_{F_k}$
\STATE $B_{k+1}\leftarrow B_k,\ F_{k+1}\leftarrow F_k$
\FOR{$i\in F_k$}
\IF{$y_i>0$}
\STATE $B_{k+1}\leftarrow B_{k+1}\cup \{i\}$,\quad $F_{k+1}\leftarrow F_{k+1}\setminus \{i\}$
\ENDIF
\ENDFOR
\STATE $k\leftarrow k+1$
\ENDWHILE
\STATE $x_{F_k}\leftarrow y$
\STATE $L^\ast_{i\mathcal{N}(i)_k}\leftarrow x_k$ ($k=1,\dots,d$)
\STATE $L^\ast_{ii}\leftarrow -\bm 1^\top x$
\ENDFOR
\ENSURE $L^\ast$
\end{algorithmic}
\caption{Active Set Algorithm (for Problem \ref{prob_1})}
\label{alg_1_full}
\end{algorithm}

In Appendix A, we show that there exists an instance that achieves the bound $O(d_i^2)$ derived in Corollary \ref{com_alg_1}. However, as demonstrated in Section \ref{Exp}, Algorithm \ref{alg_1_full} terminates in a small number of iterations and works well in practice.

\begin{remark}
As mentioned already, in Theorem 3.4 of \cite{Kunisch2003-jz}, Kunisch and Rendl showed that the update of the active set is at most $2^d$ times in a general setting: minimize $\frac{1}{2}x^\top Qx+b^\top x$ subject to $x\leq c$ ($Q\in \R^{d\times d}$ is a positive definite matrix, $b,c\in\R^d$).
However, we showed in Theorem \ref{alg1_terminate} that for our specific $Q$, the active set increases monotonically (see step 8 in Algorithm \ref{alg_1}) and stops within $d$ times. This monotonicity was proved based on Theorem \ref{main_thm1} by exploiting the special structure of our specific $Q$, namely equation (\ref{eq_g}). We also showed that, due to the simplicity of $Q$, the matrix-inverse vector multiplication $Q^{-1}b$ can be calculated in $O(d)$ time, which suggests that our active set approach is highly efficient for our problem.
\end{remark}

\subsection{Proposed Algorithm 2} \label{Sec_Alg2}
In this section, we derive another algorithm that computes the optimal graph Laplacian, with the improved computational complexity, under the assumption that 
\begin{align}
b_1\geq b_2\geq\dots\geq b_d \label{assumption_b}    
\end{align}
 without loss of generality. The algorithm directly computes the KKT point of Problem \ref{prob_3}. We derive the explicit solution of the KKT conditions \cite{Boyd2004-qd}, and provide a simple algorithm to compute the KKT point.

\begin{definition}[KKT point] \label{Def_KKT}
The KKT points of Problem \ref{prob_3} are defined as the points $(x^\ast,\lambda^\ast)\in \mathbb{R}^d\times \mathbb{R}^d$ that satisfy the following KKT-conditions:
\begin{align}
\label{KKT-grad}Q_dx^\ast +b+\lambda^\ast&=\bm 0,\\
\label{KKT-x}x^\ast &\leq \bm 0,\\
\label{KKT-l}\lambda^\ast &\geq \bm 0,\\
\label{KKT-comp}x^{\ast\top}\lambda^\ast &= 0.
\end{align}
\end{definition}

The KKT conditions are necessary and sufficient for $x^*$ to be the minimizer, since Problem \ref{prob_3} is a convex quadratic programming problem \cite{nocedal2006numerical}.

To introduce another algorithm for solving Problem \ref{prob_3}, we prove the following lemma to characterize the KKT point. 
\begin{lemma}
\label{k_lem}
Define the cumulative sum of $b$ as
\begin{align*}
S_i:=\sum_{j=1}^i b_j\quad(i=1,\dots,d)
\end{align*}
and
\begin{align*}
b_0:=+\infty,\ S_0:=0,\ b_{d+1}:=-\infty,\ S_{d+1}:=S_d.    
\end{align*}
Then, there exists exactly one index $k\in \{0,1,\dots,d\}$ such that
\begin{align}
\label{ineq_bk1}
b_k&\geq \frac{S_k}{k+1},\\
\label{ineq_bk2}
b_{k+1}&<\frac{S_{k+1}}{(k+1)+1}.
\end{align}
In other words, there is a unique $k\in \{0,1,\dots,d\}$ that simultaneously satisfies \eqref{ineq_bk1} and \eqref{ineq_bk2}.
\end{lemma}
\begin{proof}
The inequalities
$b_0=+\infty \geq 0=\frac{S_0}{0+1}$ and 
$b_{d+1}=-\infty < \frac{S_{d+1}}{d+2}$
hold. Thus, there exists at least one index $k$ that satisfy \eqref{ineq_bk1} and \eqref{ineq_bk2}, because we have assumed \eqref{assumption_b}. 
Let $k_0$ be the first $k$ to satisfy the inequalities, i.e., for $k=0,\ldots, k_0-1$, the inequalities \eqref{ineq_bk1} and \eqref{ineq_bk2} are not satisfied simultaneously.

We show that $k_0$ is the only index that simultaneously satisfies \eqref{ineq_bk1} and \eqref{ineq_bk2}.
For any $k\in \{k_0+2,\dots,d\}$, $S_k/(k+1)$ can be seen as the average value of $k+1$ real numbers $\{b_k,\ S_{k-1}/k, \dots, \ S_{k-1}/k\}$. Assume that
\begin{equation}
b_{k-1} < \frac{S_{k-1}}{k}\label{asump}
\end{equation}
holds.
Then, $b_k\leq b_{k-1}<S_{k-1}/k$ and the average of $\{b_k,\ S_{k-1}/k, \dots, \ S_{k-1}/k\}$ is strictly larger than $b_k$. 
Therefore, we can derive
$b_k<\frac{S_k}{k+1}$.
By the definition of $k_0$, inequality (\ref{asump}) holds with $k=k_0+2$, and the proof is complete by induction.
\end{proof}
\begin{theorem}
\label{thm_KKT_sol}
Let $k_0$ be the index that satisfies \eqref{ineq_bk1} and \eqref{ineq_bk2} in Lemma \ref{k_lem}. The following $(x^\ast,\lambda^\ast)$ is the KKT-point of Problem \ref{prob_3}.
\begin{align}
\label{def_xast}
x^\ast_i&:=\left\{\begin{array}{ll}\frac{1}{2}\left(\frac{S_{k_0}}{1+k_0}-b_i\right)&(i=1,\dots,k_0)\\0&(i=k_0+1,\dots,d),\end{array}\right.
\\
\label{def_last}
\lambda^\ast_i&:=\left\{\begin{array}{ll}0&(i=1,\dots,k_0)\\\frac{S_{k_0}}{1+k_0}-b_{i}&(i=k_0+1,\dots,d).\end{array}\right.
\end{align}
\end{theorem}
\begin{proof}
It is obvious that the complementary condition \eqref{KKT-comp} holds. 
By a direct calculation, we can confirm that
\eqref{KKT-grad} holds.
From the definition of $k_0$, we have
\begin{align}
x_{k_0}^\ast = \frac{1}{2}\left(\frac{S_{k_0}}{1+k_0}-b_{k_0}\right)\leq 0. \label{x_k_0}    
\end{align}
The assumption \eqref{assumption_b}, the definition \eqref{def_xast}, and \eqref{x_k_0} yield $x_{1}^\ast\leq \cdots\leq x_{k_0}^\ast\leq 0$. 
Thus, \eqref{KKT-x} holds.

To show (\ref{KKT-l}), it suffices to prove 
\begin{align}
\lambda_{k_0+1}^\ast \geq 0, \label{lambda_ineq}    
\end{align}
since 
$\lambda_{k_0+1}^\ast \leq \cdots \leq \lambda_{d}^\ast$
due to the assumption \eqref{assumption_b} and the definition \eqref{def_last}.
By the definition of $k_0$, 
\begin{equation}
\label{ineq_k01}
b_{k_0+1}<\frac{S_{k_0+1}}{k_0+2}.
\end{equation}
The right-hand side of (\ref{ineq_k01}) can be seen as the average of $k_0+2$ real numbers $\{b_{k_0+1}, S_{k_0}/(k_0+1),\dots,S_{k_0}/(k_0+1)\}$. Therefore, (\ref{ineq_k01}) implies
$b_{k_0+1}<\frac{S_{k_0}}{k_0+1}$,
and \eqref{lambda_ineq} holds.
\end{proof}

It is worth noting that $x^\ast$ defined in (\ref{def_xast}) is the unconstrained minimizer of Problem \ref{prob_4} with the free variable set $F=\{1,\dots,k_0\}$ and the active variable set $B=\{k_0+1,\dots,d\}$. That is, 
\begin{align}\label{sort_to_F}
x^\ast_F&=-Q_{|F|}^{-1}b_{F} = -\frac{1}{2}\left(\mathrm{I}_{k_0}-\frac{1}{1+k_0}\mathrm{J}_{k_0}\right)b_{F},\\
x^\ast_B&=\bm 0_{|B|}.\label{sort_to_B}
\end{align}
Here, we used (\ref{Qd_inv}) for the inverse matrix. This fact indicates that if the assumption \eqref{assumption_b} holds, Algorithm \ref{alg_1} will terminate with the free set $F={1,\dots,k_0}$ and the active set $B={k_0+1,\dots,d}$. For a general unsorted vector $b$ for Problem \ref{prob_3}, the assumption \eqref{assumption_b} can be satisfied by sorting the elements of $b$.

Now, since $x^\ast$ in Theorem \ref{thm_KKT_sol} is the global optimal solution of Problem \ref{prob_3}, we provide Algorithm \ref{alg_2} as a method to compute $x^\ast$. In Algorithm \ref{alg_2}, we first sort the given vector $b$. Here, $I\leftarrow$argsort\_descend$(b)$ represents the order of indices when the elements of $b$ are sorted in descending order, resulting in $b_{I_1}\geq b_{I_2}\geq\cdots\geq b_{I_d}$. The for loop is designed to to find the index $k_0$ used in Theorem \ref{thm_KKT_sol}. In the final step, we calculate the optimal solution $x^\ast$ with $F=\{I_1,\dots, I_{k_0}\}$, using \eqref{sort_to_F} and \eqref{sort_to_B}.

\stepcounter{algnum}
\begin{algorithm}[t]
\begin{algorithmic}[1]
\REQUIRE $d\in \mathbb{N}$, $b\in \mathbb{R}^d$
\STATE $S\leftarrow 0, F\leftarrow \emptyset$
\STATE $x^\ast\leftarrow \bm 0_{d}$, $I\leftarrow \mathrm{argsort\_descend}(b)$
\FOR{$i=1,\dots,d$}
\STATE $S\leftarrow S+b_{I_i}$
\IF{$b_{I_i}<S/(i+1)$}
\BREAK
\ENDIF
\STATE $F\leftarrow F\cup\{I_i\}$
\ENDFOR
\STATE $x^\ast_F \leftarrow -Q_{|F|}^{-1}b_F$
\ENSURE $x^\ast$
\end{algorithmic}
\caption{Calculate KKT Point by Sorting (for Problem \ref{prob_3})}
\label{alg_2}
\end{algorithm}

The computational complexity of Algorithm \ref{alg_2} is given as follows.
\begin{corollary}\label{com_alg_2}
Algorithm \ref{alg_2} solves Problem \ref{prob_3} with $O(d_i\log{d_i})$ time. 
\end{corollary}
\begin{proof}
The time complexity of sorting a size-$d_i$ vector is $O(d_i\log{d_i})$. The rest part of Algorithm \ref{alg_2} can be calculated in linear time of $d_i$.
\end{proof}

Although the output $x^\ast$ in Algorithm \ref{alg_2} corresponds to the essential parts of the optimal graph Laplacian, it is not a solution for Problem \ref{prob_1}.
Therefore,
 we present Algorithm \ref{alg_2_full} as the complete version for solving Problem \ref{prob_1}. 
Additionally, we provide the following result, which is a direct consequence of Corollary \ref{com_alg_2}.

\begin{corollary}\label{com_alg_2_full}
Algorithm \ref{alg_2_full} solves Problem \ref{prob_1} with $O(\sum_{i=1}^n d_i\log{d_i})$ time.
\end{corollary}

\begin{algorithm}[t]
\begin{algorithmic}[1]
\REQUIRE $A\in \mathbb{R}^{n\times n}$, neighbor set of each node: $\mathcal{N}(i)~(i=1,\dots,n)$
\STATE $L^\ast \leftarrow \bm 0_{n\times n}$
\FOR{$i=1,\dots,n$}
\STATE $S\leftarrow 0$, $F\leftarrow \emptyset$, $d\leftarrow |\mathcal{N}(i)|$
\STATE $x\leftarrow \bm 0_{d}$, $b\leftarrow \begin{bmatrix}2A_{ii}-2A_{i\mathcal{N}(i)_1},\dots, 2A_{ii}-2A_{i\mathcal{N}(i)_{d}}\end{bmatrix}^\top$
\STATE $I\leftarrow \mathrm{argsort\_descend}(b)$
\FOR{$j=1,\dots,d$}
\STATE $S\leftarrow S+b_{I_{j}}$
\IF{$b_{I_j}<S/(j+1)$}
\BREAK
\ENDIF
\STATE $F\leftarrow F\cup\{I_j\}$
\ENDFOR
\STATE $x_F \leftarrow -Q_{|F|}^{-1}b_F$
\STATE $L^\ast_{i\mathcal{N}(i)_k}\leftarrow x_k$ ($k=1,\dots,d$)
\STATE $L^\ast_{ii}\leftarrow -\bm 1^\top x$
\ENDFOR
\ENSURE $L^\ast$
\end{algorithmic}
\caption{Calculate KKT Point by Sorting (for Problem \ref{prob_1})}
\label{alg_2_full}
\end{algorithm}


\subsection{Existing Algorithms} \label{Sec_existing}
In this section, we explain two existing optimization algorithms for solving Problem \eqref{prob_3}. In Section \ref{Exp}, we numerically compare these algorithms with the proposed algorithms.

\subsubsection{Interior Point Method}

Problem \eqref{prob_3} can be solved using an interior point method, because the problem is equivalent to
\begin{framed}
\begin{problem}
\label{prob_IP}
\begin{equation*}
\begin{aligned}
\min_{\tilde{x} \in \mathbb{R}^{d}} \quad & \frac{1}{2}\tilde{x}^\top Q_{d} \tilde{x} - b^\top \tilde{x}\\
\mathrm{subject\ to} \quad& \tilde{x}\geq \bm 0_{d}.
\end{aligned}
\end{equation*}
\end{problem}
\end{framed}

For Problem \ref{prob_IP}, which is a special case described in \cite[Section 16.6]{nocedal2006numerical},
the perturbed KKT conditions are given by
\begin{align}
    F(\tilde{x},\lambda):= \begin{bmatrix}
        Q\tilde{x}-\lambda -b \\
        \tilde{X}\Lambda {\bf 1}_d - \sigma \mu {\bf 1}_d
    \end{bmatrix} = 0, \label{perturbed_KKT}
\end{align}
where $\tilde{X}:={\rm diag}\{\tilde{x}_1,\ldots, \tilde{x}_d\}$, $\Lambda:={\rm diag}\{\lambda_1,\ldots, \lambda_d\}$, and $\sigma\in (0,1)$.
By fixing the value of $\mu$ and applying Newton's method to the perturbed KKT conditions \eqref{perturbed_KKT}, we obtain
\begin{align}
    \begin{bmatrix}
        Q & -\mathrm{I}_d \\
        \Lambda & \tilde{X}
    \end{bmatrix}
    \begin{bmatrix}
        \Delta \tilde{x} \\
        \Delta \lambda
    \end{bmatrix}
    =
    \begin{bmatrix}
        0 \\
        -\tilde{X}\Lambda {\bf 1}_d + \sigma \mu {\bf 1}_d
    \end{bmatrix}. \label{Newton_eq}
\end{align}
Here, we assume that the pair $(\tilde{x},\lambda)$ represents a primal-dual strictly feasible point, meaning that $Q\tilde{x}-\lambda-b=0$, $\tilde{x}>0$, and $\lambda>0$.
For instance, a primal-dual strictly feasible point can be given by the following expressions:
\begin{align}
    \tilde{x} = {\rm abs}(b) + {\bf 1}_d,\quad \lambda = Q\tilde{x}-b, \label{primal-dual_feasible}
\end{align}
where ${\rm abs}(b)$ denotes the vector whose elements are the absolute values of the corresponding elements of $b$.

From Newton equation \eqref{Newton_eq}, we have
\begin{align}
     \Delta\lambda &= Q\Delta \tilde{x}, \label{Newton_eq2} =2\Delta \tilde{x} +2\left(\sum_{i=1}^d \Delta \tilde{x}_i\right){\bf 1}_d,\\
     (\Lambda +\tilde{X}Q)\Delta \tilde{x} &=  -\tilde{X}\Lambda {\bf 1}_d + \sigma \mu {\bf 1}_d. \label{Newton_eq3}
\end{align}
Multiplying the left side of \eqref{Newton_eq3} by $\tilde{X}^{-1}$, we obtain
\begin{align}
    (Q+D)\Delta \tilde{x} = t, \label{Delta_x}
\end{align}
where
\begin{align}
    D:= \tilde{X}^{-1}\Lambda,\quad t:= -\Lambda {\bf 1}_d + \sigma \mu \tilde{X}^{-1} {\bf 1}_d.
\end{align}
From \eqref{Newton_eq2},
$\Delta \lambda$ can be calculated in $O(d)$ time.
Moreover,
we can calculate $\Delta \tilde{x}$ in $O(d)$ time using Proposition \ref{Sherman-Morrison}
to \eqref{Delta_x}.
In fact, Proposition \ref{Sherman-Morrison} and \eqref{def_Q} yield
\begin{align*}
    (Q+D)^{-1} = ( 2 \mathrm{I}_d +D)^{-1} -\frac{2}{1+{\bf 1}_d^\top ( 2 \mathrm{I}_d +D)^{-1}{\bf 1}_d}( 2 \mathrm{I}_d +D)^{-1}{\bf 1}_d{\bf 1}_d^\top ( 2 \mathrm{I}_d +D)^{-1}.
\end{align*}
Because $2 \mathrm{I}_d +D$ is a diagonal matrix whose diagonal elements are positive, 
the computations of $( 2 \mathrm{I}_d +D)^{-1} t$ and $( 2 \mathrm{I}_d +D)^{-1} {\bf 1}_d$ are finished in $O(d)$ time.

Algorithm \ref{alg_IP} describes an interior point method for solving Problem \ref{prob_3},
where, for example, a primal-dual strictly feasible point $(\tilde{x}^0, \lambda^0)$ is given by \eqref{primal-dual_feasible}.
This algorithm terminates with the average complementary gap below $\varepsilon$, i.e., $(\tilde{x}^k)^\top \lambda^k/d\leq \varepsilon$, in 
\begin{align}
k=O(d\log (1/\varepsilon)), \label{IP_complexity}
\end{align}
as shown in \cite
[Theorem 3.2]{gondzio2013convergence}.

\renewcommand{\thealgorithm}{\arabic{algnum}}
\stepcounter{algnum}
\begin{algorithm}[t]
\begin{algorithmic}[1]
\REQUIRE A primal-dual strictly feasible point $(\tilde{x}^0, \lambda^0)$, stopping parameter $\varepsilon>0$, step size $\alpha>0$, shrinkage rates $\sigma, \rho\in(0,1)$, iterative number $k\leftarrow 0$
\WHILE{$(\tilde{x}^k)^\top \lambda^k/d\geq \varepsilon$}
\STATE 
Calculate $\Delta \tilde{x}^k$ by solving \eqref{Delta_x} with $\mu:=(\tilde{x}^k)^\top \lambda^k/d$
\STATE
Calculate $\Delta \lambda^k$ using \eqref{Newton_eq2}
\STATE
$(\tilde{x}^{k+1}, \lambda^{k+1}) \leftarrow (\tilde{x}^{k}, \lambda^{k}) + \alpha (\Delta \tilde{x}^k, \Delta \lambda^k)$
\WHILE{$\tilde{x}^{k+1}\not>0$ or $\lambda^{k+1}\not>0$}
\STATE $\alpha \leftarrow \rho\alpha$
\STATE $(\tilde{x}^{k+1}, \lambda^{k+1}) \leftarrow (\tilde{x}^{k}, \lambda^{k}) + \alpha (\Delta \tilde{x}^k, \Delta \lambda^k)$
\ENDWHILE
\STATE $k\leftarrow k+1$ 
\ENDWHILE
\ENSURE{$x^k = -\tilde{x}^k$}
\end{algorithmic}
\caption{Interior point method (for Problem \ref{prob_3})}
\label{alg_IP}
\end{algorithm}


\subsubsection{V-FISTA}
Problem \ref{prob_3} can be solved by using
the fast iterative shrinkage-thresholding algorithm (FISTA) \cite{Beck2009-qd}, which is a fast proximal gradient method for minimizing a composite convex function 
\begin{align*}\min_{x\in \mathbb{R}^{d}}\{F(x):=f(x)+g(x)\},\end{align*}
where $f$ and $g$ satisfy the following assumptions.
\begin{itemize}
\item $f:\mathbb{R}^d \to (-\infty,\infty)$ is convex and $\beta$-smooth for some $\beta>0$.
\item $g:\mathbb{R}^d\to (-\infty,\infty]$ is proper, closed, and convex.
\end{itemize}
Problem \ref{prob_3} is a special case of this problem by letting
$f(x) := \frac{1}{2}x^\top Qx+b^\top x$
and
$g(x):=
\begin{cases}0&(x\in C_d)\\\infty&(x\notin C_d)\end{cases}$, where $C_d$ is defined as in \eqref{C_d}.
The smoothness of $f$ is confirmed by, for any $x,y\in \mathbb{R}^d$,
\begin{align*}
\|\nabla f(x)-\nabla f(y)\|=\|Q(x-y)\|\leq \lambda_{\max} (Q)\|x-y\|,
\end{align*}
where $\lambda_{\max} (Q)$ denotes the maximum eigenvalue of $Q$.
From Lemma \ref{lem_Qd_spd},
$\lambda_{\max} (Q)=2+2d$.

V-FISTA \cite{beck2017first}, a variant of FISTA shows an improved convergence rate under the additional assumption:
\begin{itemize}
\item $f$ is $\sigma$-strongly convex for some $\sigma>0$.
\end{itemize}
For Problem \ref{prob_3} in this paper,
\begin{align*}
f(x)-2\cdot\frac{1}{2}\|x\|^2= \frac{1}{2}x^\top (2\mathrm{J}_d) x-b^\top x,
\end{align*}
which is convex.
Thus, $f$ is $\sigma$-strongly convex \cite[Theorem 5.17]{beck2017first} for $\sigma:=2$. 

The general form of V-FISTA is shown in Algorithm \ref{alg_VFISTA}. Algorithm \ref{alg_VFISTA} deals with the non-smooth term using the proximal operator, defined as
\begin{equation}
\mathrm{prox}_g(y):= \mathop{\mathrm{argmin}}_{x\in \mathbb{R}^d}\left\{g(x)+\frac{1}{2}\|x-y\|^2\right\}.
\end{equation}
In our problem, the proximal operator is given by
\begin{align}
\mathrm{prox}_{g/\beta}\left(y^k-\frac{1}{\beta}\nabla f\left(y^k\right)\right)
=\Pi_{C_d}\left(y^k-\frac{1}{\beta}\left(Qy^k+b\right)\right),\label{prox_project}
\end{align}
where $\Pi_{C_d}$ is the projection onto $C_d$, defined as $\Pi_{C_d}(x)=\{\min \{x_i,0\}\}_{i=1}^d$.

As shown in \cite{beck2017first}, the convergence rate of V-FISTA is
\begin{align}
F(x^k)-F_\mathrm{opt}=O((1-\sqrt{\sigma/\lambda_{\max}(Q)})^k)=O((1-1/\sqrt{1+d})^k), \label{convergen_rate_FISTA}
\end{align}
which is faster than the convergence rate $O(1/k^2)$ of the general FISTA algorithm and preserves the convergence rate of the restarted FISTA. Here, $F_\mathrm{opt}$ denotes the optimal objective value. Additionally, V-FISTA is simple in the sense that it does not require consideration of stepsize strategies.
From \eqref{convergen_rate_FISTA}, the iteration number $k$, which satisfies $F(x^k)-F_{\mathrm{opt}}< \varepsilon$, can be estimated with
\begin{align}
k=O\left(\frac{1}{\log \frac{\sqrt{1+d}}{\sqrt{1+d}-1}} \log (1/\varepsilon)\right). \label{iter_FISTA}
\end{align}

\renewcommand{\thealgorithm}{\arabic{algnum}}
\stepcounter{algnum}
\begin{algorithm}[t]
\begin{algorithmic}[1]
\REQUIRE $x^0, y^0\in C_d$, $\kappa=\frac{\beta}{\sigma}$, $\varepsilon>0$, $k\leftarrow 0$
\WHILE{$F(x^k)-F_{\rm opt}\geq \varepsilon$}
\STATE $x^{k+1}\leftarrow\mathrm{prox}_{g/\beta}\left(y^k-\frac{1}{L}\nabla f\left(y^k\right)\right)$
\STATE $y^{k+1}\leftarrow x^{k+1}+\left(\frac{\sqrt{\kappa}-1}{\sqrt{\kappa+1}}\right)\left(x^{k+1}-x^k\right)$
\STATE $k\leftarrow k+1$
\ENDWHILE
\ENSURE{$x^k$}
\end{algorithmic}
\caption{V-FISTA (for Problem \ref{prob_3})}
\label{alg_VFISTA}
\end{algorithm}

\subsection{Comparision} \label{Sec_comparision}

Table \ref{tab:time_complexity_comparison} compares Algorithm \ref{alg_1_full}, Algorithm \ref{alg_2_full}, an interior point method based on Algorithm \ref{alg_IP}, and a V-FISTA based on Algorithm \ref{alg_VFISTA} in terms of time complexities for solving Problem \ref{prob_1} and finite termination properties.
Here, the finite termination means whether the number of iterations required to obtain the optimal solution is finite.

\begin{table}[t]
    \centering
    \caption{Comparison of time complexities and finite termination properties for four algorithms.}
    \begin{tabular}{lcc}
        \toprule
        Algorithm & Complexity & Finite Termination \\
        \midrule
        Algorithm \ref{alg_1_full} & $O\left(\displaystyle\sum_{i=1}^n d_i^2\right)$ & Yes \\
        Algorithm \ref{alg_2_full} & $O\left(\displaystyle\sum_{i=1}^n d_i\log{d_i}\right)$ & Yes \\
        Interior point method & $O\left(\displaystyle\sum_{i=1}^n d_i^2 \log (1/\varepsilon)\right)$ & No \\
        V-FISTA & $O\left(\displaystyle\sum_{i=1}^n \frac{d_i}{\log \frac{\sqrt{1+d_i}}{\sqrt{1+d_i}-1}} \log (1/\varepsilon)\right)$ & No \\
        \bottomrule
    \end{tabular}\label{tab:time_complexity_comparison}
\end{table}

The time complexities of Algorithms \ref{alg_1_full} and \ref{alg_2_full} follow from Corollaries \ref{com_alg_1_full} and \ref{com_alg_2_full}, respectively.
For both the interior point method and V-FISTA, used for solving Problem \ref{prob_1} with Algorithms \ref{alg_IP} and \ref{alg_VFISTA} respectively, the time complexity is derived from the following steps:
\begin{itemize}
\item Problem \ref{prob_1} is divided into $n$ subproblems, each corresponding to Problem \ref{prob_3}.
\item Each instance of Problem \ref{prob_3} can be solved in $O(d)$ time, as discussed in Section \ref{Sec_existing}, with the iteration number for the interior point method and V-FISTA given by \eqref{IP_complexity} and \eqref{iter_FISTA}, respectively.
\end{itemize}

The finite termination property of Algorithms \ref{alg_1_full} and \ref{alg_2_full} is established based on the discussions in Sections \ref{Sec_Alg1} and \ref{Sec_Alg2}. In contrast, the interior point method and V-FISTA do not possess this property. Instead, they output approximate solutions determined by a parameter $\varepsilon$ in Algorithms \ref{alg_IP} and \ref{alg_VFISTA}, respectively, where the interpretation of $\varepsilon$ varies between the algorithms as detailed in Section \ref{Sec_existing}.

According to Table \ref{tab:time_complexity_comparison}, Algorithm \ref{alg_2_full} is expected to perform the best. In fact, this is demonstrated through numerical experiments in Section \ref{Exp}.

\section{Loopy Graph Laplacians}\label{DG}
In this section, we generalize the nearest graph Laplacian problem to loopy Laplacians that correspond to directed graphs with self-loops. We prove theorems that determine whether the weight of the self-loop edge of the optimal graph is positive or zero. We show that the optimal solution is easily obtained when the self-loop has a positive weight. 
Otherwise, the self-loop weight is promised to be $0$, and we can use the proposed algorithms in Section \ref{DSG}. 

\subsection{Problem Formulation} \label{Sec_pro_formulation2}
Our problem is formulated as the following Problem \ref{prob_1_sl}. 
\begin{framed}
\begin{problem}\label{prob_1_sl}
Given a graph structure $(V,E)$ and a matrix $A\in \mathbb{R}^{n\times n},$
\begin{equation*}
\begin{aligned}
\min_{L\in \mathbb{R}^{n\times n}} \quad & \|A-L\|_\mathrm{F}^2\\
\mathrm{subject\ to}\quad& L\in \mathcal{L}_d(V,E).
\end{aligned}
\end{equation*}
\end{problem}
\end{framed}
\noindent Here, $\mathcal{L}_d(V,E)$ is defined in (\ref{def_Ld}). The difference from Problem \ref{prob_1} is that the row sum of the rows that correspond to the nodes with self-loops can be positive. For rows without self-loops, we can solve the row-wise problem by the proposed algorithms in Section \ref{DSG}.\par
Thus, we consider the rows with self-loops. Let us assume that a self-loop exists in the $i$-th row, i.e., $\{i,i\}\in E$. Our problem can be written in a row-wise form as follows:
\begin{framed}
\begin{problem}
\label{prob_5}
\begin{equation*}
\begin{aligned}
\min_{l\in \mathbb{R}^{n}} \quad & \sum_{j=1}^n \left(A_{ij}-l_j\right)^2\\
\mathrm{subject\ to}\quad& l_i \geq 0,\quad
l_j \leq 0\quad(j\in\mathcal{N}(i)),\\
&l_j = 0 \quad(j\notin \mathcal{N}(i), i\neq j),\quad
{\bf 1}_n^\top l \geq 0.
\end{aligned}
\end{equation*}
\end{problem}
\end{framed}
\subsection{Determining Weights of Self-Loops}
Let $l^\ast\in \mathbb{R}^n$ be the optimal solution to Problem \ref{prob_5} and $A'_{i1},\dots, A'_{in}$ be defined as
\begin{equation}
\label{def_Aprime}
A'_{ij}:=
\begin{cases}
\max\{0,A_{ii}\} & (i=j)\\
\min\{0,A_{ij}\}& (j\in\mathcal{N}(i))\\
0 & (j\notin\mathcal{N}(i),\ i\neq j)
\end{cases}.
\end{equation}
\begin{theorem}
\label{thm_self-loop}
If $\sum_{j=1}^nA'_{ij}\geq 0$ (i.e., if $\begin{bmatrix}A'_{i1}&\dots&A'_{in}\end{bmatrix}^\top$ is a feasible point of Problem \ref{prob_5}), then $l^\ast = \begin{bmatrix}A'_{i1}&\dots&A'_{in}\end{bmatrix}^\top$.
\end{theorem}
\begin{proof}
For any feasible $l\in \mathbb{R}^n$ of Problem \ref{prob_5},
\begin{align*}
\sum_{j=1}^n \left(A_{ij}-l_j\right)^2&=\left(A_{ii}-l_i\right)^2+\sum_{j\in \mathcal{N}(i)}\left(A_{ij}-l_j\right)^2+\sum_{j\notin \mathcal{N}(i), i\neq j}\left(A_{ij}-l_j\right)^2\\
&\geq \left(A_{ii}-A'_{ii}\right)^2+\sum_{j\in \mathcal{N}(i)}\left(A_{ij}-A'_{ij}\right)^2+\sum_{j\notin \mathcal{N}(i), i\neq j}\left(A_{ij}-0\right)^2\\
&=\sum_{j=1}^n\left(A_{ij}-A'_{ij}\right)^2.
\end{align*}
\end{proof}
We show that when $A'_{ij}$ is not feasible, the self-loop weight is equal to $0$.
To this end, we prepare the following.
\begin{lemma}
\label{lem_self_increase_i}
If $\sum_{j=1}^nA'_{ij}<0$, then $l^\ast_i> A'_{ii}$.
\end{lemma}
\begin{proof}
Assume that $l_i^\ast \leq A'_{ii}$ holds. There exists at least one $k\in [n]\backslash\{i\}$ that $l_k^\ast > A'_{ik}$ holds (otherwise, $\bm 1^\top l^\ast\leq \sum_{j=1}^n A'_{ij}<0$). We consider the two cases (i) $l^\ast_i<A'_{ii}$ and (ii) $l^\ast_i=A'_{ii}$.\par
\noindent(i) If $l^\ast_i<A'_{ii}$, then $A_{ii}=A'_{ii}$ and $A_{ik}=A'_{ik}$ hold from $0\leq l_i^\ast < A'_{ii}=\max\{0,A_{ii}\}$ and $\min\{0, A_{ik}\}=A'_{ik}<l_k^\ast \leq 0$. Now let $\Delta$ be $\Delta:=\min\{A_{ii}-l^\ast_i, l^\ast_k -A_{ik}\}>0$ and $x\in \mathbb{R}^n$ be
$x_j:=
\begin{cases}
l^\ast_i +\Delta & (j=i)\\
l^\ast_k -\Delta & (j=k)\\
l^\ast_j & (\mathrm{otherwise})
\end{cases}$.
Since $l^\ast$ is a feasible point of Problem \ref{prob_5}, $x$ is also feasible. Now, from $l^\ast_i <l_i^\ast+\Delta \leq A_{ii}$ and $A_{ik}\leq l^\ast_k-\Delta<l^\ast_k$, we derive
\begin{align*}
&\sum_{j=1}^n\left(A_{ij}-l^\ast_j\right)^2-\sum_{j=1}^n\left(A_{ij}-x_j \right)^2\\
&= \left(A_{ii}-l^\ast_i\right)^2-\left(A_{ii}-x_i\right)^2+\left(A_{ik}-l^\ast_k\right)^2-\left(A_{ik}-x_k\right)^2\\
&=\left(A_{ii}-l^\ast_i\right)^2-\left(A_{ii}-(l^\ast_i+\Delta)\right)^2+\left(A_{ik}-l^\ast_k\right)^2-\left(A_{ik}-(l^\ast_k-\Delta)\right)^2 >0,
\end{align*}
which contradicts the optimality of $l^\ast$.\par
\noindent
(ii) If $l^\ast_i=A'_{ii}$, then $A_{ik}=A'_{ik}$ holds from $\min\{0, A_{ik}\}=A'_{ik}<l_k^\ast \leq 0$. Now, let $\Delta$ be $\Delta:= \frac{1}{2}(l^\ast_k-A'_{ik})=\frac{1}{2}(l^\ast_k-A_{ik})>0$, and $x\in \mathbb{R}^n$ be
$x_j:=
\begin{cases}
l^\ast_i +\Delta & (j=i)\\
l^\ast_k -\Delta & (j=k)\\
l^\ast_j & (\mathrm{otherwise})
\end{cases}$.
Since $l^\ast$ is a feasible point of Problem \ref{prob_5}, $x$ is also feasible. Now, we derive
\begin{align*}
&\sum_{j=1}^n\left(A_{ij}-l^\ast_j\right)^2-\sum_{j=1}^n\left(A_{ij}-x_j \right)^2\\
&= \left(A_{ii}-l^\ast_i\right)^2-\left(A_{ii}-x_i\right)^2+\left(A_{ik}-l^\ast_k\right)^2-\left(A_{ik}-x_k\right)^2\\
&=\left(A_{ii}-l^\ast_i\right)^2-\left(A_{ii}-(l^\ast_i+\Delta)\right)^2+\left(A_{ik}-l^\ast_k\right)^2-\left(A_{ik}-(l^\ast_k-\Delta)\right)^2\\
&=0-\Delta^2 + 4\Delta^2-\Delta^2
=2\Delta^2>0,
\end{align*}
which contradicts the optimality of $l^\ast$.
\end{proof}
\begin{theorem}
\label{thm_no_self-loop}
If $\sum_{j=1}^nA'_{ij}<0$, then $\bm 1^\top l^\ast = 0$ (i.e., the self-loop weight of the optimal Laplacian is $0$).
\end{theorem}
\begin{proof}
Assume that $\bm 1^\top l^\ast >0$. From Lemma \ref{lem_self_increase_i}, we have $ l^\ast_i>A'_{ii}$. Let $\Delta>0$ be $\Delta:= \min\{\bm 1^\top l^\ast, l^\ast_i-A'_{ii}\}$ and $x\in \mathbb{R}^n$ be
\begin{equation*}
x_j:=
\begin{cases}
l^\ast_i -\Delta & (j=i)\\
l^\ast_j & (\mathrm{otherwise})
\end{cases}.
\end{equation*}
Since $\bm 1^\top x =\bm 1^\top l^\ast-\Delta\geq 0$ and $x_j\leq 0$ for any $j\in[n]\backslash\{i\}$, we have $x_i\geq 0$ and thus $x$ is feasible. From the definition of $\Delta$, we have $A_{ii}\leq A'_{ii}\leq l^\ast_i-\Delta<l_i^\ast$. Now we can derive
\begin{align*}
\sum_{j=1}^n\left(A_{ij}-l^\ast_j\right)^2-\sum_{j=1}^n\left(A_{ij}-x_j \right)^2
&= \left(A_{ii}-l^\ast_i\right)^2-\left(A_{ii}-x_i\right)^2\\
&= \left(A_{ii}-l^\ast_i\right)^2-\left(A_{ii}-(l^\ast_i-\Delta)\right)^2 >0,
\end{align*}
which contradicts the optimality of $l^\ast$.
\end{proof}
\begin{corollary}
$\bm 1^\top l^\ast>0$ is equivalent to $\sum_{j=1}^nA'_{ij}>0$. 
\end{corollary}
\begin{proof}
From Theorem \ref{thm_self-loop}, we have
\begin{align*}
\sum_{j=1}^nA'_{ij}>0\quad &\Rightarrow\quad \bm 1^\top l^\ast=\sum_{j=1}^nA'_{ij}>0, \\
\sum_{j=1}^nA'_{ij}=0\quad &\Rightarrow\quad \bm 1^\top l^\ast=\sum_{j=1}^nA'_{ij}=0.
\end{align*} 
 Combining these facts and Theorem \ref{thm_no_self-loop}, we obtain $\bm 1^\top l^\ast>0\,\Leftrightarrow\, \sum_{j=1}^nA'_{ij}>0$. 
\end{proof}
\subsection{Proposed Algorithm for Loopy Laplacians}
We propose Algorithm \ref{alg_3} as a method to compute the nearest loopy Laplacian. 
\renewcommand{\thealgorithm}{\arabic{algnum}}
\stepcounter{algnum}
\begin{algorithm}[t]
\begin{algorithmic}[1]
\REQUIRE $A\in \mathbb{R}^{n\times n}$, $(V,E)$, neighbor set of each node: $\mathcal{N}(i)~(i=1,\dots,n)$
\STATE $L^\ast\leftarrow \bm 0_{n\times n}$
\FOR{$i=1,\dots,n$}
\IF{$\{i,i\}\notin E$}
\STATE $L^\ast_{\{i\}[n]}\leftarrow$ solution of Algorithm \ref{alg_1} or \ref{alg_2} applied to row $i$.
\ELSE 
\STATE Define $a':=\begin{bmatrix}A'_{i1}&\dots&A'_{in}\end{bmatrix}^\top$ as in (\ref{def_Aprime}).
\IF {$\bm 1^\top a'\geq 0$}
\STATE $L^\ast_{\{i\}[n]}\leftarrow a'$
\ELSE 
\STATE $L^\ast_{\{i\}[n]}\leftarrow$ solution of Algorithm \ref{alg_1} or \ref{alg_2} applied to row $i$, by assuming $\{i,i\}\notin E$.
\ENDIF
\ENDIF
\ENDFOR
\ENSURE $L^\ast$

\end{algorithmic}
\caption{Compute the nearest Loopy Laplacian}
\label{alg_3}
\end{algorithm}
The for loop calculates the optimal solution for each row. If row $i$ does not have a self-loop, we can solve it as in the loop-less case. If row $i$ has a self-loop and the assumption of Theorem \ref{thm_self-loop} ($\sum_{j=1}^nA'_{ij}\geq0$) holds, then (\ref{def_Aprime}) is the optimal solution to row $i$. Otherwise (if $\sum_{j=1}^nA'_{ij}<0$),  from Theorem \ref{thm_no_self-loop}, we know that the edge weight of the self-loop is $0$ and therefore we can calculate as in the loop-less case by assuming $\{i,i\}\notin E$.

Note that we can also use Algorithms \ref{alg_IP} and \ref{alg_VFISTA} as explained in Section \ref{Sec_existing} in steps 4 and 10 of Algorithm \ref{alg_3} in place of Algorithms \ref{alg_1} and \ref{alg_2}.

\section{Numerical Experiments}\label{Exp}
In this section, we numerically compare our proposed Algorithms \ref{alg_1_full} and \ref{alg_2_full} with the interior point method based on Algorithm \ref{alg_IP} and V-FISTA based on Algorithm \ref{alg_VFISTA}.
For the numerical comparison, we used the following parameter values: for Algorithm \ref{alg_IP}, $\varepsilon=10^{-6}$, $\alpha=1$, $\rho=0.9$, and $\sigma = 0.5$; for Algorithm \ref{alg_VFISTA}, $\varepsilon=10^{-6}$, $\beta = 2+2\max\{d_1,\ldots, d_n\}$, and $\sigma =2$. 
Note that as explained in Section \ref{Sec_comparision}, Algorithms \ref{alg_IP} and \ref{alg_VFISTA} only output approximate solutions characterized by the parameter $\varepsilon$, unlike Algorithms \ref{alg_1_full} and \ref{alg_2_full}, which provide exact optimal solutions.
In this comparison, we focus solely on loop-less cases, as the algorithms for loop-less cases can also address loopy ones, as demonstrated in Algorithm \ref{alg_3}.
All the tests were computed by 
MATLAB R2023b on a Windows 10 Pro with Intel Xeon Silver 4214R CPU @ 2.40 GHz and 192GB RAM.

\subsection{Directed Loop-Less Graph Laplacians}\label{DLLLGL}
This section shows the computational time of the nearest graph Laplacian problem when the graph structure is directed. 

We generated matrix $A$ by constructing a ``noisy'' graph Laplacian. We employed the method used in \cite{sato2019optimal}.
\begin{description}
\item[Step 1.] Generate an unweighted and undirected graph structure $(V,E)$ by the Watts-Strogatz model \cite{Watts1998-pu}.
\item[Step 2.] Replace every edge with two bidirectional edges.
\item[Step 3.] For any $\{i,j\}\in E$, set the edge weight $w_{ij}\leftarrow 10\times \mathrm{rand}$.
\item[Step 4.] Construct the graph Laplacian of $(V,E,w)$ and denote by $X\in\mathbb{R}^{n\times n}$.
\item[Step 5.] Construct $A\in \mathbb{R}^{n\times n}$ by
\begin{align}
A = X +5\times {\rm randn}(n,n).
\end{align}
\end{description}
Here, $\mathrm{rand}$ is a random scaler drawn from the uniform distribution in the interval $(0,1)$, and $\mathrm{randn}(n,n)$ is a random matrix whose elements are drawn from the standard normal distribution. To avoid memory shortage, we use the sparse matrix format. The Watts-Strogatz model randomly generates a graph that represents both high clustering properties and small path length properties, which can be seen in social networks \cite{Albert2002-gp}.

\subsubsection{Small-scale case ($n=100$)}

We considered small-scale cases with $n=100$ with the average out-degree $20$.
These cases arise in the context of model reduction problems, as explained in Section \ref{Intro_reduction}.

Fig. \ref{fig:ExecutionTimes_N100_K10}
illustrates the box plot when
we measured the computational time for 100 different random $A$.
This figure displays a box plot comparing the execution times of four different algorithms: Active Set (Algorithm \ref{alg_1_full}), Sort (Algorithm \ref{alg_2_full}), IP (Algorithm \ref{alg_IP}), and V-FISTA (Algorithm \ref{alg_VFISTA}). Execution times are plotted on a logarithmic scale in seconds on the vertical axis. The central line of each box represents the median execution time, the edges of the box are the first and third quartiles, and the whiskers extend to show the range of the data, excluding outliers. Outliers are plotted individually as red plus signs. Algorithms \ref{alg_1_full} and \ref{alg_2_full} show similar distributions of execution times, both with median values around $10^{-2}$ seconds. Algorithms \ref{alg_IP} and \ref{alg_VFISTA} have higher median execution times, around $10^{-1}$ seconds, indicating that they are generally slower than Algorithms \ref{alg_1_full} and \ref{alg_2_full}. The spread of the data points suggests that the variability in execution times is greater for Algorithms \ref{alg_IP} and \ref{alg_VFISTA} compared to Algorithms \ref{alg_1_full} and \ref{alg_2_full}.

\begin{figure}[t]
    \centering
    \includegraphics[width=12cm]{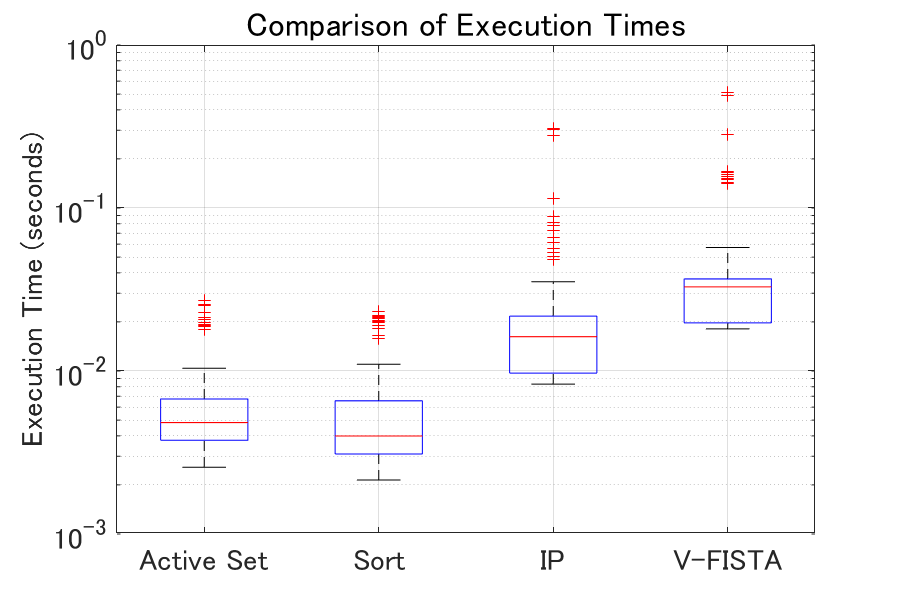}
    \caption{Comparative analysis of algorithm execution times for small-scale cases.}
\label{fig:ExecutionTimes_N100_K10}
\end{figure}

\subsubsection{Large-scale case ($n=30000$)}

We considered large-scale cases with $n=30000$ with the average out-degree $20$.
These cases arise in the context of large-scale system identification problems, as explained in Section \ref{Intro_identification}.

Fig.\,\ref{fig:ExecutionTimes_N30000_K10} represents the execution times of the same four algorithms: Active Set (Algorithm \ref{alg_1_full}), Sort (Algorithm \ref{alg_2_full}), IP (Algorithm \ref{alg_IP}), and V-FISTA (Algorithm \ref{alg_VFISTA}). This plot depicts the execution times in actual seconds on the vertical axis, ranging from approximately 75 to 110 seconds. The median execution times for Algorithms \ref{alg_1_full} and \ref{alg_2_full} are clustered closely around 80 seconds, with relatively small interquartile ranges, indicating a tight grouping of data and less variability in execution times. The IP (Algorithm \ref{alg_IP}) displays a median execution time slightly above 100 seconds, with a broader interquartile range, suggesting greater variability. The V-FISTA (Algorithm \ref{alg_VFISTA}) exhibits a median execution time that is slightly slower than those of Algorithms \ref{alg_1_full} and \ref{alg_2_full}, as well as a wider spread of execution times, indicated by a larger interquartile range.

\begin{figure}[t]
    \centering
    \includegraphics[width=12cm]{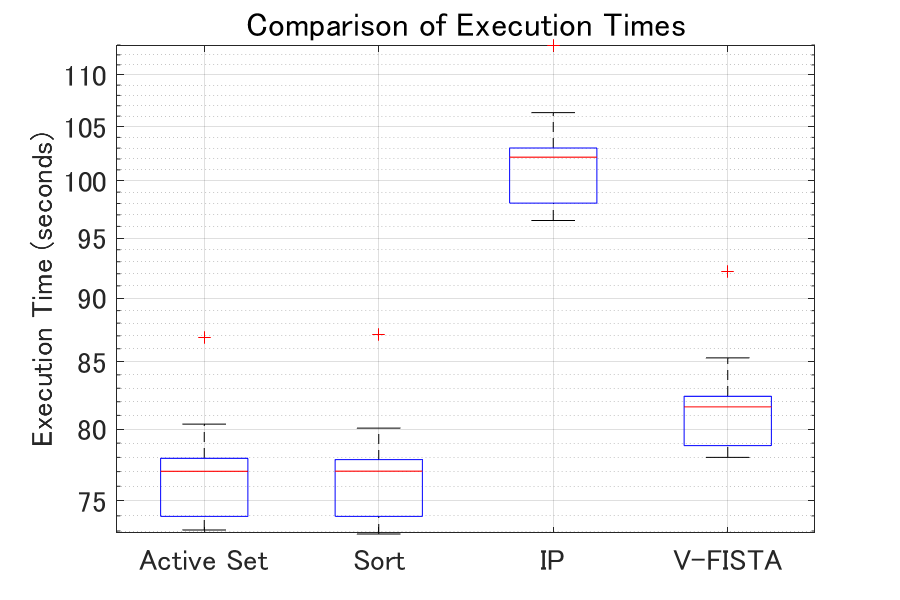}
    \caption{Comparative analysis of algorithm execution times for large-scale cases.}
\label{fig:ExecutionTimes_N30000_K10}
\end{figure}

\subsection{Worst Case for Algorithm 1}
In Section \ref{DLLLGL}, we demonstrated that the performance of Algorithms \ref{alg_1_full} and \ref{alg_2_full} is nearly identical, although the theoretical computational complexity of Algorithm \ref{alg_2_full} is better than that of Algorithm \ref{alg_1_full}, as shown in Table \ref{tab:time_complexity_comparison} in Section \ref{Sec_comparision}. In this section, we present the computational time for an artificially generated worst-case scenario applied to Algorithm \ref{alg_1_full}, which is detailed in Appendix \ref{ape_A}.

The worst-case instance of Problem \ref{prob_3} is the case when the free variable set always decreases only one variable at a time. First we generated the sequence $b_1, b_2, \dots$, by 
$b_1 = -\frac{1}{2}$,
$b_{k} = (k+1)b_{k-1}-S_{k-1}$ $(k=2,3,\dots)$,
where $S_k$ is the cumulative sum: $S_k=\sum_{i=1}^k b_k$ (see Appendix A).  Then, we defined $A\in \mathbb{R}^{n\times n}$ by (\ref{Aworst}).
 It is worth pointing out that, from \eqref{app10}, the sequence $\{A_{i\mathcal{N}(i)_k}\}_{k=1}^d$ increases exponentially, and such instances for large $d$ is unlikely to be seen in the real world.

Figs.\,\ref{fig:ExecutionTimes(worst)_N100_K10} and \ref{fig:ExecutionTimes(worst)_N30000_K10} show the computational times with $n=100$ and $n=30000$, respectively.
Because of the way the matrix $A$ was constructed, as described above,
Algorithm \ref{alg_1_full} exhibited slower execution times compared to Algorithm \ref{alg_2_full}.
Notably, the interior point method, referred to as IP in these figures, was outperformed by the other methodologies. 

\begin{figure}[t]
    \centering
    \includegraphics[width=12cm]{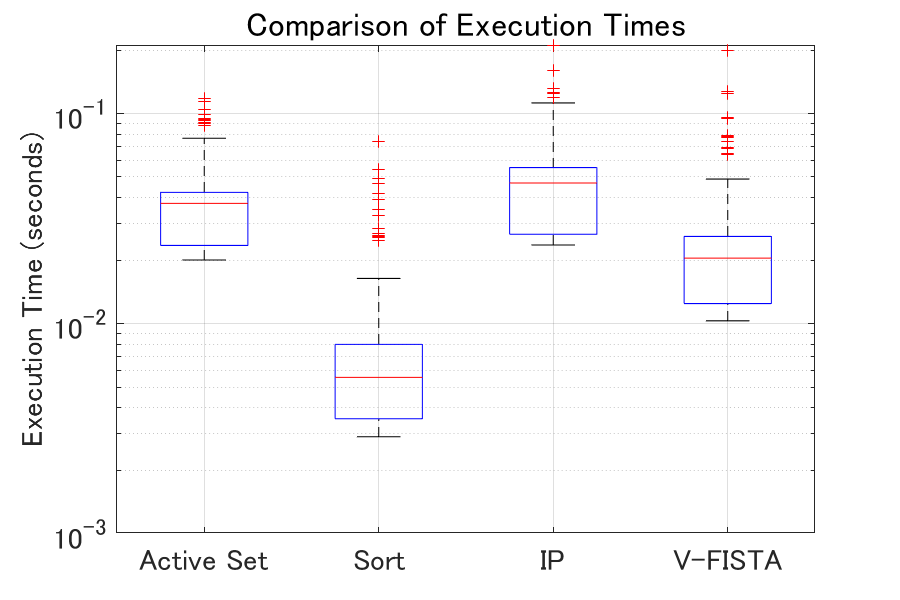}
    \caption{Comparative analysis of algorithm execution times for small-scale worst-case scenarios on Algorithm \ref{alg_1_full}.}
\label{fig:ExecutionTimes(worst)_N100_K10}
\end{figure}

\begin{figure}[t]
    \centering
    \includegraphics[width=12cm]{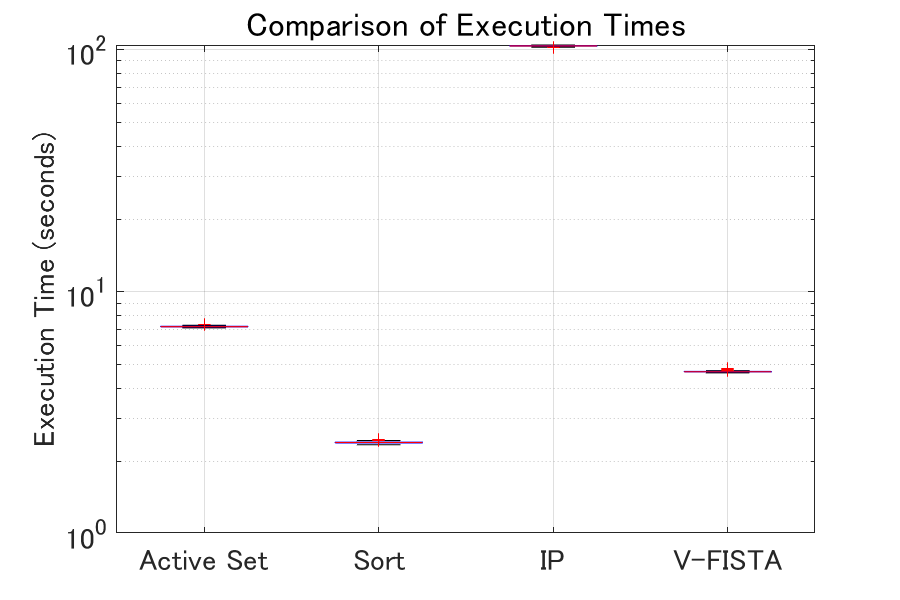}
    \caption{Comparative analysis of algorithm execution times for large-scale worst-case scenarios on Algorithm \ref{alg_1_full}.}
\label{fig:ExecutionTimes(worst)_N30000_K10}
\end{figure}

\section{Concluding Remarks}\label{Conclusion}
In this paper, we formulated the problem of finding the nearest graph Laplacian matrix in the Frobenius norm as a convex quadratic optimization problem with linear constraints. In the case of directed simple graphs, we proposed two novel algorithms that directly compute the global optimal solution to this problem. 
We showed that from a computational complexity perspective, our proposed algorithms are more efficient than both the interior point method and V-FISTA.
We also proved that the proposed methods can be generalized to directed graphs with self-loops by a simple preprocessing step.  Moreover, the numerical experiments corroborated the results of our theoretical analysis.
Furthermore, we demonstrated that there are worst-case scenarios for  Algorithm \ref{alg_1_full}, leading us to recommend Algorithm \ref{alg_2_full} for broader applicability.

Finally, we address the nearest graph Laplacian problem for undirected simple graphs.  
Consider Problem \ref{prob_1} with the additional constraint $L=L^\top$.
The problem can be reformulated as Problem \ref{prob_2} with the additional constraint $M_{\mathrm{sym}}x ={\bm 0}_m$. This is achieved by replacing each undirected edge $\{i,j\}$ with two bidirectional directed edges $\{i,j\}$ and $\{j,i\}$, and imposing the constraint $w_{ij}=w_{ji}$ on the weights of every pair-edge.
Here, $m':=2m=2|E|$ is the number of edges after the doubling. The matrix $M_{\mathrm{sym}}\in \mathbb{R}^{m\times m'}$ is defined as follows: For the $k$-th undirected edge in $E$ being $\{i,j\}$ $(i<j)$, with $L_{ij}$ and $L_{ji}$ corresponding to $x_u$ and $x_v$, we define $(M_\mathrm{sym})_{ku}:= 1$, $(M_\mathrm{sym})_{kv}:= -1$, and all other entries as $0$.
Although the matrix $M_{\mathrm{sym}}$ plays the role of symmetrizing the edge weight, the constraint $M_{\mathrm{sym}}x = \bm{0}_m$ prevents the problem from being divided into row-wise simple and smaller problems.
Similar difficulties arise in the case of Problem \ref{prob_1_sl} with the additional constraint $L = L^\top$.
Therefore, developing efficient algorithms for the nearest graph Laplacian problem for undirected graphs remains a challenge for future work.

\bibliographystyle{siamplain}
\bibliography{references}

\appendix
\section{Worst-Case Instance of Algorithm 1} \label{ape_A}
The worst case of Algorithm \ref{alg_1} is when the number of free variables decreases one by one in each loop.
Without loss of generality, we assume that index $k$ was added to the active set in loop $d-k+1~(k=1,\dots,d)$, that is,
\begin{align*}
|F_1|&=\{1,\dots,d\},\,
\ldots,\,
|F_{d-k+1}|=\{1,\dots, k\},\\
|F_{d-k+2}|&=\{1,\dots,k-1\},\,
\ldots,\,
|F_d| =\{1\},\,
|F_{d+1}|=\emptyset.
\end{align*}

\begin{proposition}$\{b_i\}_{i=1}^d$ is the worst case of Algorithm \ref{alg_1} if and only if, 
\begin{align}
b_1&<0,\label{app0}\\
b_k&\leq (k+1)b_{k-1}-S_{k-1}\quad(k=2,\dots,d),
\label{app-2}
\end{align}
where $S_k$ is the cumulative sum $S_k:=\sum_{i=1}^kb_i$.
\end{proposition}
\begin{proof}
First, we assume that $\{b_i\}_{i=1}^d$ is the worst case of Algorithm \ref{alg_1}. Since index $1$ was added to the active set in loop $d$, we have
$-(Q_{|F_d|})^{-1}b_1=-\frac{1}{4}b_1>0$,
and thus (\ref{app0}) is derived. In loop $d-k+1$ ($k=2,\dots,d$), since index $k$ is added to the active set and indices $1,\dots,k-1$ are not, we have
\begin{align}
\left(-Q_{|F_{d-k+1}|}^{-1}b_{F_{d-k+1}}\right)_k=\left(-Q_{k}^{-1}b_{[k]}\right)_k&>0,\label{app1}\\
\left(-Q_{|F_{d-k+1}|}^{-1}b_{F_{d-k+1}}\right)_j=\left(-Q_{k}^{-1}b_{[k]}\right)_j &\leq 0 \quad (j=1,\dots,k-1).\label{app2}
\end{align}
From Lemma \ref{lem_Qd_spd}, we have 
\begin{align}
\left(-Q_{k}^{-1}b_{[k]}\right)_j&=\left(-\frac{1}{2}\left(\mathrm{I}_{k}-\frac{1}{1+k}\mathrm{J}_{k}\right)b_{[k]}\right)_j=\frac{1}{2}\left(-b_j+\frac{S_k}{1+k}\right).\label{app3}
\end{align}
Using (\ref{app3}), inequalities (\ref{app1}) and (\ref{app2}) are equivalent to, for any $k=2,\dots,d$,
\begin{align}
b_k&<\frac{S_k}{1+k},\label{app4}\\
b_j&\geq \frac{S_k}{1+k}\quad (j=1,\dots,k-1).\label{app5}
\end{align}
From (\ref{app5}) with $j=k-1$, we have
\begin{equation}
b_{k-1}\geq \frac{S_k}{1+k}\quad(k=2,\dots,d).\label{app6}
\end{equation}
Using (\ref{app6}) and $S_k=S_{k-1}+b_k$, (\ref{app-2}) is derived and the necessary condition is proved.\par
Next, we prove that any $\{b_i\}_{i=1}^d$ satisfying (\ref{app0}) and (\ref{app-2}) is a worst case. To this end, proving (\ref{app4}) and (\ref{app5}) is sufficient. From (\ref{app0}) and (\ref{app-2}), we have
\begin{align}
b_2\leq 3b_1-b_1=2b_1<b_1<0.\label{app7}
\end{align}
Now, if $0>b_1>\dots>b_{k-1}$ holds for some $k\geq 3$, then
\begin{align}
b_k\leq (k+1)b_{k-1}-S_{k-1}<(k+1)b_{k-1}-(k-1)b_{k-1}=2b_{k-1}<b_{k-1}.\label{app8}
\end{align}
Combining (\ref{app7}) and (\ref{app8}), we have
\begin{equation}
0>b_1>b_2>\dots>b_d, \label{app9}
\end{equation}
and 
\begin{equation}
b_k\leq 2b_{k-1}\leq\dots\leq 2^{k-1}b_1. \label{app10}
\end{equation}
From (\ref{app-2}) and $S_k=S_{k-1}+b_k$, we can derive (\ref{app6}), and (\ref{app5}) follows from (\ref{app9}). From $S_{k-1}>(k-1)b_{k-1}$ and $2b_{k-1}\geq b_{k}$, we have
$S_{k-1}>(k-1)b_{k-1}\geq \frac{k-1}{2}b_k>kb_k$,
where the last inequality is by $\frac{k-1}{2}<k~(k\geq 2)$ and $b_{k}<0$. Thus, (\ref{app4}) is derived.
\end{proof}

For the one block problem of row $i$ (Problem \ref{prob_3}), vector $b^{(i)}\in \mathbb{R}^{d_i}$ is generated by
$b^{(i)} = \begin{bmatrix}2A_{ii}-2A_{i\mathcal{N}(i)_1}&\dots&2A_{ii}-2A_{i\mathcal{N}(i)_{d_i}}\end{bmatrix}^\top$.
Thus, an inefficient instance for Algorithm \ref{alg_1_full} in the form of matrix $A\in \mathbb{R}^{n\times n}$ can be generated by
\begin{align}
A_{ij} = 
\begin{cases}
-\frac{1}{2}b_k&(j=\mathcal{N}(i)_k)\\
0 & (\mathrm{otherwise}).
\end{cases}\label{Aworst}
\end{align}
\end{document}